\documentclass[11pt]{amsart}

\usepackage{amsthm, amsfonts, amssymb, amscd}
\usepackage[pagebackref,colorlinks]{hyperref}
\usepackage{tikz-cd}
\usepackage{amscd}
\usepackage{geometry}
\usepackage{marginnote}

\theoremstyle{definition}
\newtheorem{ntn}{Notation}[section]

\theoremstyle{plain}
\newtheorem{lem}[ntn]{Lemma}
\newtheorem{prp}[ntn]{Proposition}
\newtheorem{thm}[ntn]{Theorem}
\newtheorem{cor}[ntn]{Corollary}

\theoremstyle{definition}

\newtheorem{exa}[ntn]{Example}

\numberwithin{equation}{section}

\newcommand{\N}{\mathbb{N}}
\newcommand{\z}{\mathbb{Z}}
\newcommand{\q}{\mathbb{Q}}

\newcommand{\C}{\mathbb{C}}
\newcommand{\Pp}{\mathbb{P}}
\newcommand{\F}{\mathbb{F}}

\newcommand{\EE}{\mathcal{E}}
\newcommand{\Aa}{\mathcal{A}}
\newcommand{\BB}{\mathcal{B}}
\newcommand{\GG}{\mathcal{G}}
\newcommand{\RR}{\mathcal{R}}
\newcommand{\WW}{\mathcal{W}}
\newcommand{\II}{\mathcal{I}}
\newcommand{\RP}{\mathcal{RP}}
\newcommand{\PP}{\mathcal{P}}
\newcommand{\RB}{\mathcal{RB}}
\newcommand{\OO}{\mathcal{O}}
\newcommand{\KK}{\mathcal{K}}
\newcommand{\LL}{\mathcal{L}}

\newcommand{\ppp}{\mathfrak{p}}
\newcommand{\mmm}{\mathfrak{m}}

\newcommand{\ff}{{F^\times}}

\newcommand{\rr}{{R^\times}}
\renewcommand{\aa}{{A^\times}}

\newcommand{\tors}{{{\rm Tor}_1^{\z}}}

\newcommand{\Lan}{\langle \langle}
\newcommand{\Ran}{\rangle \rangle}
\newcommand{\half}{{\Big[\frac{1}{2}\Big]}}

\newcommand{\mt}{\mapsto}
\newcommand{\lan}{\langle}
\newcommand{\ran}{\rangle}
\newcommand{\se}{\subseteq}
\newcommand{\arr}{\rightarrow}
\newcommand{\larr}{\longrightarrow}
\newcommand{\harr}{\hookrightarrow}
\newcommand{\two}{\twoheadrightarrow}

\newcommand{\stabe}{{\rm Stab}}
\newcommand{\GL}{\mathit{{\rm GL}}}

\newcommand{\SL}{\mathit{{\rm SL}}}

\newcommand{\Ind}{{\rm Ind}}
\renewcommand{\char}{{\rm char}}

\newcommand{\coker}{{\rm coker}}
\newcommand{\im}{{\rm im}}
\newcommand{\ind}{{\rm ind}}
\newcommand{\Hom}{{\rm Hom}}

\newcommand{\inc}{{\rm inc}}
\newcommand{\id}{{\rm id}}

\newcommand{\Aut}{{\rm Aut}}
\newcommand{\tor}{{\rm tors}}
\newcommand{\sep}{{\rm sep}}
\newcommand{\Gal}{{\rm Gal}}

\newcommand {\mtx}[2]
{\left(\!\!\!
\begin{array}{cc}
#1    \\
#2 
\end{array}
\!\!\!\right)}

\newcommand {\mtxx}[4]
{\left(\!\!
\begin{array}{cc}
\!\!#1 & \!\!#2   \\
\!\!#3 & \!\!#4
\end{array}
\!\!\!\right)}

\newtheoremstyle{athm}
  {}
  {}
  {\itshape}
  {}
  {\scshape}
  {}
  {.5em}
  {\thmnote{#3}}
\theoremstyle{athm}
\newtheorem*{athm}{}

\begin{document}

\title[The homology of $\SL_2$ of discrete valuation rings]
{The homology of $\SL_2$ of discrete valuation rings 
}
\author{Kevin Hutchinson, Behrooz Mirzaii, Fatemeh Y. Mokari}
\address{School of Mathematics and Statistics,
 University College Dublin, Belfield, Dublin 4, Ireland}
\email{kevin.hutchinson@ucd.ie}
\address{Institute of Mathematical and Computer Sciences (ICMC), University of Sao Paulo, Sao Carlos, Brazil}
\email{bmirzaii@icmc.usp.br}
\address{Institute of Mathematical and Computer Sciences (ICMC), University of Sao Paulo, Sao Carlos, Brazil}
\email{f.mokari61@gmail.com}
\begin{abstract}
Let $A$ be a discrete valuation ring with field of fractions $F$ and (sufficiently large) residue field $k$. We prove that 
there is a natural exact sequence
\[
\begin{array}{c}
H_3(\SL_2(A),\z\half)\to H_3(\SL_2(F),\z\half)\to \RP_1(k)\half\to 0,
\end{array}
\]
where $\RP_1(k)$ is the refined scissors congruence group of $k$. Let $\Gamma_0(\mmm_A)$ denote the congruence 
subgroup consisting of matrices in $\SL_2(A)$ whose lower off-diagonal entry lies in the maximal ideal $\mmm_A$. We 
also prove that there is an exact sequence
\[
\begin{array}{c}
0\to \overline{\PP}(k)\half\to H_2(\Gamma_0(\mmm_A),\z\half)\to H_2(\SL_2(A),\z\half)\to I^2(k)\half\to 0,
\end{array}
\]
where $I^2(k)$ is the second power of the fundamental ideal of the Grothendieck-Witt ring $\mathrm{GW}(k)$ and 
$\overline{\PP}(k)$ is a certain quotient of the  scissors congruence group (in the sense of Dupont-Sah)  $\PP(k)$ of $k$.  
\end{abstract}
\maketitle

\section*{Introduction}
The purpose of this article is to study the low-dimensional homology of certain linear groups and to demonstrate what we 
hope is interesting behaviour. 

Let $A$ be a discrete valuation ring with field of fractions $F$ and residue field $k$. $k$ may be finite, but for the validity of 
our proofs should be \emph{sufficiently large}: if $|k|=p^d$, then $d(p-1)>6$. Our first main result (Theorem \ref{mainthm} below) 
is:
\begin{athm}[{\bf Theorem A}.]\label{thm1}
There is a natural exact sequence 
\[
\begin{array}{c}
H_3(\SL_2(A),\z\half) \arr H_3(\SL_2(F),\z\half) \arr \RP_1(k)\half \arr 0,
\end{array}
\]
where $\RP_1(k)$ is the \emph{refined scissors congruence group} of the field $k$. 
\end{athm}

Dupont and Sah defined the scissors congruence group $\PP(F)$ of a field $F$ (also called the \emph{pre-Bloch group} of $F$) in 
\cite{dupont-sah1982}. It is an abelian group given by generators and relations.  They  related $\PP(\C)$ and its subgroup $\BB(\C)$ to
$H_3(\SL_2(\C),\z)$. In fact, there is a natural isomorphism $H_3(\SL_2(\C),\z)\cong K_3^\ind(\C)$, the indecomposable $K_3$ of $\C$. 
Some time later, Suslin showed in \cite{suslin1991} how to generalize the result of Dupont and Sah to arbitrary (infinite) fields, identifying 
$K_3^\ind(F)$ with $\BB(F)$ modulo a certain well-understood torsion subgroup (for a precise statement, see Theorem \ref{bloch-wigner} 
below). 

However, when the field $F$ is not quadratically closed, the natural surjective homomorphism $H_3(\SL_2(F),\z)\to K_3^\ind(F)$ has a 
nontrivial, and often quite large, kernel (which we denote by $H_3(\SL_2(F),\z)_0$). To give an analogous description of $H_3(\SL_2(F),\z)$ 
one must replace the scissors congruence group $\PP(F)$ of Dupont-Sah with the \emph{refined scissors congruence group} $\RP_1(F)$ 
(and its subgroup $\RB(F)$), as shown by the first author in \cite{hutchinson-2013}. The refined scissors congruence group of a 
commutative ring $A$ is defined by a presentation analogous to $\PP(A)$ but as a module over  the group ring 
$\RR_A:=\z[A^\times/(A^\times)^2]$ rather than as an abelian group. The homology groups $H_\bullet(\SL_2(A),\z)$ are naturally 
$\RR_A$-modules and this module structure plays a central role in all of our calculations. Thus, for example,  the exact sequence 
of {\sf Theorem~A} is a sequence of $\RR_A$-modules.

{\sf Theorem A} generalizes the main result of \cite{hutchinson2017}, where a (somewhat more precise) result was proved in the 
case of complete discrete valuations (with residue characteristic not equal $2$). We note furthermore that there is an analogous 
behaviour of second homology groups, replacing the functor $\RP_1$ with the first Milnor-Witt $K$-theory functor 
$K_1^{\mathrm{MW}}$: Suppose that $1/2\in A$ and that $k$ is infinite. Then there is a natural short exact sequence
\[
\begin{array}{c}
0\arr H_2(\SL_2(A),\z) \arr H_2(\SL_2(F),\z) \arr K_1^{\mathrm{MW}}(k) \arr 0.
\end{array}
\]
This follows from the facts that $H_2(\SL_2(R),\z)$ can be identified naturally with $K_2^{\mathrm{MW}}(R)$ for  local domains $R$ 
with infinite residue field (by Schlichting \cite{schlichting}) and that there is an exact localization sequence in Milnor-Witt $K$-theory 
\[
\begin{array}{c}
0\arr K_2^{\mathrm{MW}}(A)\arr K_2^{\mathrm{MW}}(F)\arr K_1^{\mathrm{MW}}(k) \arr 0
\end{array}
\]
for regular local rings $A$ containing $1/2$ by the main result of \cite{GSZ}. 

We do not know if we should expect the leftmost map in sequence of {\sf Theorem A} to be injective in general. It is certainly injective 
in the case of a complete discrete valuation (by \cite{hutchinson2017}). The general question is related to the Gersten conjecture for $K_3$. 

Our second main result concerns the calculation of the second homology of the congruence group 
\[
\Gamma_0(\mmm_A):= \left\{ \mtxx{a}{b}{c}{d} \in \SL_2(A)\ | \ c\in \mmm_A\right\}
\]
where $\mmm_A$ is the maximal ideal of the discrete valuation ring $A$. We show the following (for a more precise statement see 
Theorem \ref{barP} below):

\begin{athm}[{\bf Theorem B.}]\label{thm2}
The inclusion $\Gamma_0(\mmm_A)\to\SL_2(A)$ induces an exact sequence (of $\RR_A$-modules)
\[
\begin{array}{c}
0\arr\overline{\PP}(k)\half \arr H_2(\Gamma_0(\mmm_A),\z\half)\arr H_2(\SL_2(A),\z\half)\arr I^2(k)\half\arr 0.
\end{array}
\]
\end{athm}
Here $I^2(k)$ denotes the second power of the fundamental ideal $I(k)$ of the Grothendieck-Witt ring $\mathrm{GW}(k)$ 
of the field $k$ and $\overline{\PP}(k)$ is a certain quotient of the scissors congruence group $\PP(k)$. 

The map $H_2(\SL_2(A),\z)\cong K_2^{\mathrm{MW}}(A)\to I^2(A)\to I^2(k)$ is well-known from Milnor-Witt $K$-theory. We 
give an explicit formula for the map $\overline{\PP}(k)\half\to H_2(\Gamma_0(\mmm_A),\z\half)$: Let 
$d:A^\times\to \Gamma_0(\mmm_A)$ denote the inclusion $a\mapsto \mathrm{diag}(a,a^{-1})$. The map  sends a generator   
$[\bar{a}]$ of $\PP(k)$, $a\in A^\times\setminus \{1\}$,  to $1/2$ the Pontryagin product $d(a)\wedge d(1-a)$. This element in 
turn is known to map to $[a][1-a]\in K_2^{\mathrm{MW}}(A)\half\cong H_2(\SL_2(A),\z\half)$ (see, for example,
\cite[Corollary 4.2]{hutchinson2016}),
which is of course  $0$ by the Steinberg relation in Milnor-Witt $K$-theory. However, the kernel of the quotient map 
$\PP(k)\to\overline{\PP}(k)$ is in general small and often trivial. In particular, we show how to calculate $\overline{\PP}(k)$ 
in the case where $F$ is a global field (and hence $k$ is finite). For example, when $F=\q$ and $A=\z_{(p)}$ ($p\geq 11$) 
we have  $\overline{\PP}(\F_p)\half=\PP(\F_p)\half$ when $p\not\equiv 2\pmod{3}$ and this group is cyclic of order $(p+1)'$ 
where $m'$ is the odd part of the integer $m$. 
~\\
~\\
{\bf Remark.}
We have had to state all of our main results over $\z\half$ because of $2$-torsion ambiguities in several of the 
fundamental results that we rely on, and because our methods of proof (eg. the character-theoretic local-global principle) 
require us to invert $2$. We do not know what modifications we should expect to have to make to the results presented if 
we want them to be  valid over $\z$ instead. 

\subsection*{Overview of the article}
In Section \ref{scg} we review the relevant facts about scissors congruence groups and their relation to the third homology 
of $\SL_2$, and to indecomposable $K_3$. 
 
In Section \ref{sec:RSCG} we prove Theorem~A (Theorem \ref{mainthm} below). We first must define the $\RR_A$-homomorphism
$\Delta_\pi:H_3(\SL_2(F),\z)\to \RP_1(k)$ (depending on a choice of uniformizer $\pi$). Up to some known results about $K_3$, 
Theorem~A can be reduced to an exact sequence involving scissors congruence groups 
(Theorem \ref{thm:4.2} ). This in turn is proved using a character-theoretic local-global principle  for modules over group rings 
$\z[G]$ where $G$ is a (multiplicative) elementary abelian $2$-group (recalled from \cite{hutchinson-2017}). 

Our second main theorem, Theorem B (Theorem \ref{barP} below) is proved by a careful comparison of  the Mayer-Vietoris homology 
exact sequence associated to the amalgamated product  decomposition $\SL_2(F)\cong \SL_2(A)*_{\Gamma_0(\mmm_A)}\SL_2(A)$ 
with a spectral sequence relating the homology of $\SL_2(A)$ to (refined) scissors congruence groups. 
Section \ref{mv} reviews the basic facts about the  amalgamated product decomposition and the associated long exact Mayer-Vietoris 
sequence. In particular, we detail how this is a sequence of $\RR_F$-modules, because this module structure plays an essential  role 
in our subsequent calculations.

In Section \ref{connecting}, we give the technical details of the proof of Theorem B. First, we use Theorem A to identify 
$\im(\delta)$ with $\overline{\PP}(k)\half$ (as an $\RR_F$-module). Here $\delta: H_3(\SL_2(F),\z\half)\to H_2(\Gamma_0(\mmm_A),\z\half)$ 
is the connecting homomorphism of the Mayer-Vietoris sequence. The remainder of the section is devoted to the explicit formula for $\delta$ 
and the calculation of  kernel and cokernel of $H_2(\Gamma_0(\mmm_A),\z\half)\to H_2(\SL_2(A),\z\half)$. 
Both of these require calculations with spectral 
sequences associated to the action of $\GL_2(A)$ on complexes of vectors. Indeed, the connecting homomorphism $\delta$ is shown to be 
essentially identifiable with a $d^3$-differential from such a spectral sequence. 

Finally, in Section \ref{global} we calculate $\overline{\PP}(k)\half$  in the case that $F$ is a global field. We use this to obtain more 
precise calculations for the groups  $\SL_2(\z_{(p)})\subset\SL_2(\q)$. 

\subsection*{Terminology and Notation}    
In this article all rings are commutative, except possibly group rings, and have an identity element. For a ring $A$, $A^\times$ will denote 
its  group of units.  For an abelian group or module $M$,  $M\half$ denotes $M\otimes_\z\z\half$. 

~\\
{\bf Aknowledgements.} 
BM and FM would like to acknowledge the support of S\ ~ao Paulo Research Foundation FAPESP (Funda\c{c}\~ao de Amparo \`a Pesquisa 
do Estado de S\~ao Paulo) for their visit to University College Dublin during the year 2019 (Grant Numbers 2017/26310-7 and 
2018/03561-7 respectively), which part of this work is done. BM and FM also  would like to thank KH and University College Dublin for their 
hospitality and for very friendly working environment.

\section{Review: Scissors congruence groups and the third homology of  \texorpdfstring{$\SL_2$}{Lg}}\label{scg}

\subsection{Classical scissors congruence groups and a Bloch-Wigner exact sequence}

For a ring $A$, let 
\[
\WW_A:=\{a\in \aa:1-a \in \aa \}.
\]
The {\it scissors congruence group} $\PP(A)$ of $A$ is the quotient of
the free abelian group generated by symbols $[a]$, $a\in \WW_A$,
by the subgroup generated by elements
\[
X_{a,b}:=[a] -[b]+\bigg[\frac{b}{a}\bigg]-\bigg[\frac{1- a^{-1}}{1- b^{-1}}\bigg]
+ \bigg[\frac{1-a}{1-b}\bigg],
\]
where $a, b, a/b  \in \WW_A$. Let
\[
S_\z^2(\aa):=(\aa \otimes \aa)/\lan a\otimes b+b\otimes a:a,b \in \aa\ran.
\]
We denote the elements of $\PP(A)$ and $S_\z^2(\aa)$ represented by $[a]$ 
and $a\otimes b$ again by $[a]$ and $a\otimes b$, respectively. 
By direct computation one sees that 
\[
\lambda: \PP(A) \arr S_\z^2(\aa), \ \ \ \ [a] \mapsto a \otimes (1-a),
\]
is a well-defined homomorphism. 
The kernel of $\lambda$ is called the {\it Bloch group} of $A$ and is
denoted by $\BB(A)$. 

In fact we have the following: If $A$ is a field or a local ring whose residue field has more than
five elements, then we have the exact sequence
\[
0\arr \BB(A) \arr \PP(A) \arr S_\z^2(\aa) \arr K_2^M(A) \arr 0,
\]
(see \cite[Lemma 4.2]{mirzaii2017}), where $K_2^M(A)$ is the second Milnor $K$-group of $A$.

Recall that for a local ring  $A$ there is a natural homomorphism of graded commutative rings 
$K_\bullet^M(A) \to K_\bullet(A)$, from Milnor $K$-theory 
to $K$-theory. The \textit{indecomposable $K_3$ of $A$}, $K_3^\ind(A)$, is defined to be the cokernel 
of the map $K_3^M(A)\to K_3(A)$.  
Over a local ring (or more generally a ring with many units) the Bloch group
and  the indecomposable part of the third $K$-group are deeply connected.

\begin{thm}[Bloch-Wigner exact sequence]\label{bloch-wigner}
Let $A$ be either a field or a local domain whose residue field 
has at least $11$ elements. Then we have a natural exact sequence 
\[
0 \arr \tors(\mu(A),\mu(A))^\sim \arr K_3^\ind (A) \arr \BB(A) \arr 0,
\]
where $\tors(\mu(A),\mu(A))^\sim$ is the unique nontrivial extension of
$\tors(\mu(A),\mu(A))$ by $\z/2$ if $\char(A)\neq 2$ and is equal to 
$\tors(\mu(A),\mu(A))$ if $\char(A)= 2$.
\end{thm}
\begin{proof}
The case of infinite fields  has been proved in \cite{suslin1991} and the case of finite fields has been settled in \cite{hutchinson-2013}.
The case of local rings has been dealt in \cite{mirzaii2017}.
\end{proof}

A Bloch-Wigner exact sequence also exists over a ring with many units 
\cite{mirzaii2011}, \cite{mirzaii-mokari2015}.

\subsection{Refined scissors congruence groups}

Let $A$ be a ring. Let $\GG_A:=\aa/(\aa)^2$ and set $\RR_A:=\z[\GG_A]$.
The element of $\GG_A$ represented by $a \in \aa$ is denoted by $\lan a \ran$.
We set $\Lan a\Ran:=\lan a\ran -1\in \RR_A$.

Let $\RP(A)$ be the quotient of the free $\RR_A$-module generated by symbols 
$[a]$, $a\in \WW_A$, by the $\RR_A$-submodule generated by elements
\[
Y_{a,b}:=[a]-[b]+\lan a\ran\bigg[\frac{b}{a}\bigg]-
\lan a^{-1}-1\ran\Bigg[\frac{1-a^{-1}}{1-b^{-1}}\Bigg] - 
\lan 1-a\ran\Bigg[\frac{1-a}{1-b}\Bigg],
\]
where $a,b,a/b \in \WW_A$. We have a natural surjective map $\RP(A)\arr \PP(A)$ and from the definition it follows immediately that
\[
\PP(A)\simeq \RP(A)_{\GG_A}=H_0(\GG_A, \RP(A)).
\]

Let $\II_A$ be the augmentation ideal of the group ring $\RR_A$.
By direct, but tedious, computation one can show that the map
\[
\lambda_1: \RP(A) \arr \II_A^2, \ \ \ \ \ \ [a] \mapsto \Lan a \Ran\Lan 1-a\Ran
\]
is a well-defined $\RR_A$-homomorphism. If we consider $S_\z^2(\aa)$ as
a trivial $\GG_A$-module, then the map
\[
\lambda_2: \RP(A) \arr S_\z^2(\aa), \ \ \ \ \ [a] \mapsto a \otimes (1-a),
\]
is a homomorphism of $\RR_A$-modules. In fact $\lambda_2$ is
the composite $\RP(A) \arr \PP(A) \overset{\lambda}{\arr} S_\z^2(\aa)$.

The {\it refined scissors congruence group of $A$} is defined as the $\RR_A$-module 
\[
\RP_1(A):=\ker(\lambda_1:\RP(A) \arr \II_A^2).
\]
The {\it refined Bloch group} of $A$ is defined as the $\RR_A$-module 
$\RB(A):=\ker(\lambda_2|_{\RP_1(A)})$ (see \cite{hutchinson2013}, \cite{hutchinson-2013}).

\begin{prp}$($\cite[Proposition 2.9]{hutchinson2017}$)$\label{hutchinson2}
Let $A$ be a ring. Then 
\par {\rm (i)} $\RP_1(A) \arr \PP(A)$ induces the isomorphism
$
\begin{array}{c}
\RP_1(A)\bigg[\frac{1}{2}\bigg]_{\GG_A}\simeq \PP(A)\bigg[\frac{1}{2}\bigg],
\end{array}
$
\par {\rm (ii)} $\RB(A) \arr \BB(A)$ induces the isomorphism
$
\begin{array}{c}
\RB(A)\bigg[\frac{1}{2}\bigg]_{\GG_A}\simeq \BB(A)\bigg[\frac{1}{2}\bigg],
\end{array}
$
\par {\rm (iii)} $\RB(A) \arr \RP_1(A)$ induces the isomorphism
$
\begin{array}{c}
\II_A\RB(A)\bigg[\frac{1}{2}\bigg]\simeq \II_A\RP_1(A)\bigg[\frac{1}{2}\bigg].
\end{array}
$
\end{prp}

Over finite fields the Bloch group and the refined Bloch group are the same. 

\begin{prp}$($\cite[Section 7]{hutchinson-2013}$)$\label{hutchinson3}
If $k$ is a finite field (with at least $4$ elements), then $\GG_k$ acts trivially on $\RB(k)$. In 
particular $\RB(k)\simeq \BB(k)$. Moreover
\[
\begin{array}{c}
\RB(k)\half=\RP_1(k)\half\simeq \PP(k)\half=\BB(k)\half.
\end{array}
\]
\end{prp}

For more results on finite fields we refer the reader to \cite[Section 7]{hutchinson-2013}.

\subsection{A refined Bloch-Wigner exact sequence}
For any ring $A$, $H_\bullet(\SL_2(A),\z)$ is naturally a module over the ring $\RR_A$ as follows: Given $a\in A^\times$, 
choose $M\in \GL_2(A)$ with determinant $a$. Then $\lan a \ran \cdot z:= (C_M)_* (z)$ where $C_M$ denotes conjugation 
by $M$. As we will see, this module structure plays a central role in all of our calculations below.

For any local ring, whose residue field has at least three elements, there is a natural 
homomorphism of $\RR_A$-modules 
\[
H_3(\SL_2(A),\z) \arr K_3^\ind (A),
\]
(in which $K_3^\ind (A)$ has the trivial $\RR_A$-module structure).
This is surjective if $A$ is an infinite field \cite[Section 5]{hutchinson-tao2009}.

Throughout this article, we will say that a field $F$ is {\it sufficiently large} if either $F$ is infinite or $|F|=p^d$ with $(p-1)d>6$. 
Thus $F$ is sufficiently large if and only if $|F|\notin  \{2, 3, 4, 5, 7, 8, 9, 16, 27, 32,64\}$.

\begin{thm}$($\cite[Theorem 3.22]{hutchinson2017}$)$\label{refined-bloch-wigner}
Let $A$ be a local domain with sufficiently large residue field.
Then there is a commutative diagram of  $\RR_A$-modules with exact rows (where all terms in the lower row are trivial $\RR_A$-modules):
\[
\begin{tikzcd}
0 \ar[r]& \tors(\mu(A),\mu(A))\bigg[\frac{1}{2}\bigg]\ar[r]\ar[d, "="] 
& H_3(\SL_2(A), \z\bigg[\frac{1}{2}\bigg]) \ar[r]\ar[d]& \RB(A)\bigg[\frac{1}{2}\bigg] \ar[r]\ar[d] &0\\
0 \ar[r]& \tors(\mu(A),\mu(A))\bigg[\frac{1}{2}\bigg]\ar[r] &K_3^\ind(A)\bigg[\frac{1}{2}\bigg]\ar[r]&\BB(A)\bigg[\frac{1}{2}\bigg]\ar[r]&0.\\
\end{tikzcd}
\]
\end{thm}

Using this and taking $\GG_A$-coinvariants on the top row, one deduces:

\begin{prp}\label{hutchinson4}$($\cite[Corollary 3.23]{hutchinson2017}$)$
Let $A$ be a local domain with sufficiently large residue field. Then 
\[
\begin{array}{c}
H_3(\SL_2(A),\z\half)_{\GG_A} \simeq K_3^\ind (A)\half.
\end{array}
\]
In particular we have the exact sequence
\[
\begin{array}{c}
0 \arr \II_AH_3(\SL_2(A),\z\half) \arr H_3(\SL_2(A),\z\half) \arr 
K_3^\ind (A)\half \arr 0.
\end{array}
\]
\end{prp}

Note that for any $\RR_A$-module $M$, we have $\II_AM=\ker{(M\to M_{\GG_A})}$.

Let $H_3(\SL_2(A),\z)_0$ be the kernel of the map $H_3(\SL_2(A),\z) \arr K_3^\ind (A)$.
Theorem \ref{refined-bloch-wigner} and Propositions \ref{hutchinson2} and \ref{hutchinson4} imply that:

\begin{cor}\label{hutchinson5} 
Let $A$ be a local domain with sufficiently large residue field. Then
\[
\begin{array}{c}
H_3(\SL_2(A),\z\half)_0=\II_AH_3(\SL_2(A),\z\half)  \simeq \II_A\RP_1(A)\half  =\II_A\RB(A)\half.
\end{array}
\]
\end{cor}
\smallskip

\subsection{Special elements in \texorpdfstring{$\PP(A)$}{Lg}  and \texorpdfstring{$\RP(A)$}{Lg} }

Let $A$ be a  local ring. For an element $a \in \WW_A$, let $\{a\}:=[a]+[a^{-1}] \in \PP(A)$.
If $U_{1,A}:=1+\mmm_A$, then $\aa=\WW_A \cup U_{1,A}$. For $u\in U_{1,A}$, we define
$\{u\}:=\{ua\}-\{a\}$, for any $a\in \WW_A$. It is easy to see that this definition is well-defined. By a direct  
calculation one can show that the map 
\[
\psi:\aa \arr \PP(A), \ \ \  a\mapsto \{a\},
\]
is a homomorphism and $(\aa)^2$ is in its kernel, i.e. $\{ab\}=\{a\}+\{b\}$ and  $\{c^2\}=0$ for all 
$a,b,c \in\aa$. Let
\[
\widetilde{\PP}(A):=\PP(A)/\KK_A,
\]
where $\KK_A$ is the subgroup of $\PP(A)$ generated by the elements $\psi(a)=\{a\}$, $a\in \aa$.
Since $\lambda(\{a\})=(-a)\otimes a$, we have the natural homomorphism
\[
\tilde{\lambda}:\widetilde{\PP}(A) \arr \widetilde{S}_\z^2(\aa), 
\]
where $\widetilde{S}_\z^2(\aa):=S_\z^2(\aa)/\lan (-a)\otimes a|a\in \aa\ran$. We set
$\widetilde{\BB}(A):=\ker(\tilde{\lambda})$.

We now consider two different liftings to $\RP(A)$ of the family of  elements $\{a\}$ in $\PP(A)$. For $a\in \WW_A$, let
\[
\psi_1(a):=[a]+\lan -1\ran [a^{-1}] \ \ \ \ \ \text{and} \ \ \ \ \ \
\psi_2(a):=\lan 1-a\ran\Big(\lan a\ran[a]+[a^{-1}]\Big).
\]
For $u\in U_{1,A}=\aa \backslash \WW_A$, we define 
\[
\psi_i(u):=\psi_i(ua)-\lan u\ran\psi_i(a), 
\]
where $a\in \WW_a$. By \cite[Lemma 4.2]{hutchinson2017} this is independent of the choice of $a\in \WW_A$, 
so the definition is well defined. Now one can show \cite[Proposition 4.3]{hutchinson2017} that 
the maps
\[
\psi_i:\aa \arr \RP(A), \ \ \ \ a\mapsto \psi_i(a), \ \ i=1,2
\]
define $1$-cocycles; i.e. we have $\psi_i(ab)=\lan a\ran \psi_i(b)+\psi_i(a)$, for all $a,b\in \aa$. 
For basic properties of $\psi_i(a)$ see \cite[Section 4]{hutchinson2017} 
and \cite[Section 3]{hutchinson2013}.

Let $\KK_A^{(i)}$,  $i=1,2$, denote the $\RR_A$-submodule of $\RP(A)$ generated by the set 
$\{\psi_i(a)| a\in \aa\}$. One can show that 
\[
\lambda_1(\psi_i(a))=p_{-1}^+\Lan a\Ran=\Lan -a\Ran \Lan a\Ran,
\]
where $p_{-1}^+:=\lan -1\ran+1\in \RR_A$. Thus $\lambda_1(\KK_A^{(i)})=p_{-1}^+\II_A
\subseteq \II_A^2$.
Moreover $\ker(\lambda_1|_{\KK_A^{(i)}})$ is annihilated by $4$ \cite[Lemma 4.6]{hutchinson2017}. Let
\[
\widetilde{\RP}(A):=\RP(A)/\KK_A^{(1)}.
\]
Then 
\[
\tilde{\lambda}_1:\widetilde{\RP}(A) \arr \II_A^2/p_{-1}^+\II_A, \ \ \ \ \ \text{and} \ \ \ \ \ 
\tilde{\lambda}_2:\widetilde{\RP}(A) \arr \widetilde{S}_\z^2(\aa)
\]
induced by $\lambda_1$ and $\lambda_2$ respectively, are well-defined $\RR_A$-homomorphism. 
We set
\[
\widetilde{\RP}_1(A):=\ker({\tilde{\lambda}}_1), \ \ \ \widetilde{\RB}(A):=\ker{(\tilde{\lambda}}_2).
\]
It is easy to see that  $\RP(A) \arr \PP(A)$ induces the natural maps
\[
\widetilde{\RP}(A) \arr \widetilde{\PP}(A), \  \  \  \  \  \widetilde{\RP}_1(A) \arr \widetilde{\PP}(A), \  \  \  \  \ 
\widetilde{\RB}(A) \arr \widetilde{\BB}(A).
\]

\begin{prp}\label{hutchinson6}$($\cite[Corollary 4.7]{hutchinson2017}$)$
Let $A$ be a local ring such that its residue field has at least $4$ elements. Then the 
natural maps $\RP_1(A) \arr \widetilde{\RP}_1(A)$ and $\RB(A) \arr \widetilde{\RB}(A)$ are surjective 
with kernel annihilated by $4$. In particular 
\[
\begin{array}{c}
\RP_1(A)\half \simeq \widetilde{\RP}_1(A)\half, \ \ \ \RB(A)\half \simeq \widetilde{\RB}(A)\half.
\end{array}
\]
\end{prp}


Let $A$ be a local ring. As in \cite[Lemma 1.3]{suslin1991} one can show that the 
element $[a]+[1-a] \in \BB(A)$ is independent of the choice of $a\in \WW_A$. We denote this constant by
\[
c_A:=[a]+[1-a], \ \ \  a\in \WW_A.
\]
This constant has order dividing $6$.

In \cite[Lemma 4.9]{hutchinson2017}, it has been shown that 
$[a]+\lan -1\ran[1-a]+\Lan1-a\Ran \psi_1(a)$ is in  $\RB(A)$ and is independent of the choice of $a\in \WW_A$.
We denote this constant by 
\[
C_A:=[a]+\lan -1\ran[1-a]+\Lan1-a\Ran \psi_1(a), \ \ \ a\in \WW_A. 
\]
Under the homomorphism $\RP(A) \arr \PP(A)$, $C_A$ maps to $c_A$.
One can show \cite[Lemma 4.9]{hutchinson2017} that 
\[
3C_A=\psi_1(-1)\ \ \  \text{and}\ \ \  6C_A=0.
\]

\subsection{Refined scissors congruence group with generators and relations}

The following theorem gives a description of the structure of $\RP_1(A)$.

\begin{thm}\label{hutchinson7} $($\cite[Corollary 4.4]{hutchinson-2017}, \cite[Proposition 5.4]{hutchinson2017}$)$
Let $A$ be either a field with at least four elements or a local ring whose residue field has more than
$10$ elements. Then
\[
\begin{array}{c}
\RP_1(A)\half = e_{-1}^+\RP(A)\half,
\end{array}
\]
where $e_{-1}^+:=p_{-1}^+/2=(\lan -1\ran +1)/2$. In particular $\lan -1\ran\in \RR_A$ acts trivially on $\RP_1(A)\half$ (and hence also  
on $H_3(\SL_2(A),\z\half)$).
\end{thm}

It follows from this theorem that $\RP_1(A)\half $ is a quotient of $\RP(A)\half$ and hence
admits a simple explicit presentation as a $\RR_A$-module.

\begin{prp}\label{g(a)}
For a ring $A$ let $\RP'(A)$ be the $\RR_A$-module
with generators denoted by symbols $[a]'$, $a\in \WW_A$, subject to the following relations:
\par {\rm (i)} $\displaystyle[a]'-[b]'+\lan a\ran\bigg[\frac{b}{a}\bigg]'-
\lan a^{-1}-1\ran\Bigg[\frac{1-a^{-1}}{1-b^{-1}}\Bigg]'- 
\lan 1-a\ran\Bigg[\frac{1-a}{1-b}\Bigg]'=0$ for all $a,b, a/b\in \WW_A$,
\par {\rm (ii)} $\lan-1\ran [a]'=[a]'$ for all $a\in \WW_A$,
\par {\rm (iii)} $[a]'+[a^{-1}]'=0$ for all $a\in\WW_A$.\\
If $A$ is either a field with at least four elements or a local ring where
its residue field has more that $10$ elements, then the $\RR_A$-module homomorphism
\[
\begin{array}{c}
\RP'(A)\half \arr \RP_1(A)\half, \ \ \ \ [a]'\mt \frac{1}{2}g(a)
\end{array}
\]
is an isomorphism, where $g(a):=p_{-1}^+[a]+\Lan1-a\Ran \psi_1(a)$.
\end{prp}
\begin{proof}
 See \cite[Corollary 4.5]{hutchinson-2017} and
\cite[Remark 5.6]{hutchinson2017}.
\end{proof}

\subsection{The key identity}\label{key}

We recall the following \emph{key identity}:

\begin{lem}$($\cite[Theorem 3.12]{hutchinson2013}, \cite[Theorem 4.14]{hutchinson2017}$)$\label{lem:key}
Let $B$ be a field or a local ring whose residue field has at least $5$ elements. Then for any $a\in B^\times$ we have 
\[
2\Lan a \Ran C_{B}=\psi_1(a)-\psi_2(a)
\]
in $\RP(B)$.
\end{lem}

\begin{cor} \label{cor:key} 
If $B$ is a local ring whose residue field has at least $5$ elements,
then for all $a\in\WW_{B}$, we have 
\[
\Lan a \Ran C_B=\lan a-1\ran \Lan-a\Ran[a]
\]
in $\widetilde{\RP}(B)$.
\end{cor}
\begin{proof} Since $3C_B+\psi_1(-1)=0$, we have $2C_{B}=-C_{B}$ in $\widetilde{\RP}(B)$. Thus 
$\Lan a\Ran C_{B}=\psi_{2}(a)$ in $\widetilde{\RP}(B)$.
However, $0=\psi_1(a)=[a]+\lan-1\ran[a^{-1}]$ in $\widetilde{\RP}(B)$. Hence $[a^{-1}]=-\lan{-1}\ran[a]$. Thus in 
$\widetilde{\RP}(B)$ we have 
\begin{eqnarray*}
\psi_{2}(a)&=& \lan{1-a}\ran\left(\lan{a}\ran[a]+[a^{-1}]\right)\\
&=& \lan{1-a}\ran \left(\lan{a}\ran[a]-\lan{-1}\ran[a]\right)\\
&=& \lan{a-1}\ran(\lan{-a}\ran -1)[a]=\lan{a-1}\ran\Lan-a\Ran[a].
\end{eqnarray*}
\end{proof}

We will need the following refinement of \cite[Lemma 6.1]{hutchinson2017}.

\begin{lem}\label{lem:slr}
Let $B$ be a local ring with maximal ideal $\mmm_B$ and residue field $k$. 
Let $\LL_{B}$ denote the $\RR_B$-submodule of $\RP({B})$ generated by the elements
 $[au]-[a]$ and $\Lan u\Ran C_{B}$, $a\in \WW_B$, $u\in U_{1,B}=1+\mmm_B$. Then there 
 is a short exact sequence of $\RR_{B}$-modules
\[
0\to \LL_{B}\to \widetilde{\RP}(B)\to \widetilde{\RP}(k)\to 0.
\]
\end{lem}
\begin{proof}
Clearly the functorial map, $p$ say, $\widetilde{\RP}(B)\to \widetilde{\RP}(k)$ is surjective and 
$\LL_{B}\subseteq\ker(p)$. Let $Q(B):= \widetilde{\RP}(B)/\LL_B$. 

We claim that $U_{1,B}\subset B^\times$ acts trivially on $Q(B)$: Let $a\in \WW_{B}$. 
By Corollary \ref{cor:key}
\[
[a]=\lan{-a}\ran[a]-\lan{a-1}\ran \Lan{a}\Ran C_{B}.
\]
Since for any $u\in U_{1,B}$, $au\in\WW_{B}$, again by  Corollary \ref{cor:key}
\[
[au]=\lan{-au}\ran[au]-\lan{au-1}\ran\Lan au\Ran C_{B}.
\]
However in $Q(B)$ we have
\[
\lan{au-1}\ran\Lan{au}\Ran C_{B}= \lan{a-1}\ran \lan{u'}\ran \Lan{au}\Ran C_{B}= \lan{a-1}\ran\Lan{a}\Ran C_{B},
\]
where $u'= (1-au)/(1-a)\in U_{1,B}$. (Note that in above we use the formula
$\Lan{au}\Ran=\Lan{a}\Ran+\lan{a}\ran\Lan{u}\Ran$ and the fact that for all $w\in U_{1,B}$, 
$\lan{w}\ran C_{B}=C_{B}$ in $Q(B)$.) Thus in $Q(B)$
\[
0=[{a}]-[{au}]=\lan{-a}\ran [{a}]-\lan{-au}\ran [{au}]
\]
for all $a\in \WW_{B}$, $u\in U_{1,B}$. It follows that 
$\lan{-a}\ran[{a}]=\lan{-au}\ran[au]=\lan{-au}\ran[{a}]$ in $Q(B)$ for all $a\in \WW_B, u\in U_{1,B}$. Multiplying both sides by 
$\lan{-a}\ran$, we deduce that 
\[
[{a}]=\lan{u}\ran[{a}]
\] 
in $Q({B})$ for all $u\in U_{1,B}$, $a\in \WW_{B}$, proving the claim.

It follows that the $\RR_{B}$-module structure on $Q(B)$ induces a $\RR_{k}$-module structure, 
since $k^\times \simeq B^\times/U_{1,B}$. 

Thus there is a well-defined $\RR_{k}$-module homomorphism $\widetilde{\RP}({k})\to Q(B)$, 
$[{\bar{a}}]\mapsto [{a}]+\LL_{B}$,  giving an inverse to the map $\bar{p}:Q(B)\to \widetilde{\RP}({k})$. 
\end{proof}

\section{The third homology of  \texorpdfstring{$\SL_2$}{Lg} of a  discrete valuation ring}\label{sec:RSCG}

Throughout this section $A$ will be a discrete valuation ring with maximal ideal $\mmm_A$ and residue field $k$.
Let $F$ be the field of fractions of $A$ and let  $v=v_A:\ff\arr \z$ be the associated discrete valuation.
 We fix a uniformizer
$\pi$ of the valuation, i.e. a generator of $\mmm_A$. Moreover let $U_{n,A}:=1+\mmm_A^n$.

In this section we will prove the following main theorem:

\begin{thm}\label{mainthm}
Let $A$ be a discete valuation ring with  field of fractions $F$ and   sufficiently large residue field $k$.  
There is an $\RR_A$-map $\Delta_\pi: H_3(\SL_2(A),\z)\to \widetilde{\RP}_1(k)$  such that the sequence
\[
\begin{array}{c}
H_3(\SL_2(A),\z\half) \arr H_3(\SL_2(F),\z\half) 
\overset{\Delta_\pi}{\larr} \RP_1(k)\half \arr 0
\end{array}
\]
is an exact sequence of $\RR_A$-modules.
\end{thm}

\subsection{Induction of modules}

From the natural maps $A\harr F$ and $A\two k$ we obtain the homomorphisms of groups
\[
\GG_A \harr \GG_F, \ \ \ \  
\GG_A \two \GG_k. 
\]
Thus any $\RR_F$-module or any $\RR_k$-module has a natural $\RR_A$-module structure.
For any $\RR_k$-module $M$ we define the induced $\RR_F$-module
\[
\Ind_k^F M:=\RR_F\otimes_{\RR_A} M.
\]
Note that, since $M$ is an $\RR_A$-$\RR_k$-bimodule, it follows that  $\Ind_k^F M$ is an $\RR_k$-module in a natural way: 
$\lan \bar{a}\ran.(x\otimes m)=
x\otimes \lan \bar{a}\ran m=\lan a\ran x\otimes m$.

\begin{lem}$($\cite[Lemma 5.4]{hutchinson-2017}$)$\label{lem:decom}
For any $\RR_k$-module $M$, the $\RR_k$-homomorphisms
\[
\rho_0: \Ind_k^F M \arr M, \ \ \ \ \lan u\pi^r\ran \otimes m\mapsto \lan \bar{u}\ran m, 
\]
and 
\[
\rho_\pi: \Ind_k^F M \arr M, \ \ \ \lan u\pi^r\ran \otimes m \mt
\begin{cases}
\lan \bar{u}\ran m & \text{if $r$ is odd}\\
0 & \text{if $r$ is even,}
\end{cases}
\]
for all $r\in \z, u\in \aa$, induce the isomorphism of $\RR_k$-modules (and so $\RR_A$-modules)
\[
\Ind_k^F M\overset{(\rho_0,\rho_\pi)}{\overset{\simeq}{\larr}} M\oplus M, \ \ \ \ 
\II_F\Ind_k^F M\overset{(\rho_0,\rho_\pi)}{\overset{\simeq}{\larr}} \II_k M\oplus M.
\]
\end{lem}

\subsection{The specialization homomorphism}

There is a natural surjective specialization introduced and developed in \cite[Section 4.3]{hutchinson2013} 
and \cite[Section 5]{hutchinson-2017} defined as follows:
\[
S_v:\RP(F) \arr \Ind_k^F\widetilde{\RP}(k),
\ \ \ \ \ \ \ 
[a] \mapsto 
\begin{cases}
1\otimes [\bar{a}] &\text{if $v(a)=0$}\\
1\otimes C_k & \text{if $v(a)>0$}\\
-(1\otimes C_k) & \text{if $v(a)<0$}
\end{cases}
\]
This is a map of $\RR_F$-modules. It induces a well-defined map
\[
S_v:\widetilde{\RP}(F) \arr \Ind_k^F\widetilde{\RP}(k).
\]
Furthermore this induces a well-defined specialization 
\[
S_v:\RP_1(F) \arr \Ind_k^F\widetilde{\RP}_1(k).
\]

The composite
\[
\RP_1(F) \overset{S_v}{\larr} \Ind_k^F \widetilde{\RP}_1(k) 
\overset{(\rho_0,\rho_\pi)}{\overset{\simeq}{\larr}} 
\widetilde{\RP}_1(k)\oplus \widetilde{\RP}_1(k)
\]
induces two natural homomorphisms of $\RR_A$-modules
\[
\delta_\pi:=\rho_\pi\circ S_\pi:\RP_1(F) \arr \widetilde{\RP}_1(k) ,\ \ \ \ \ \  
\delta_0:=\rho_0\circ S_\pi:\RP_1(F) \arr \widetilde{\RP}_1(k).
\]
Furthermore the restriction of these maps to $\II_F\RP_1(F)$ induce homomorphisms of $\RR_A$-modules
\[
\delta_\pi:\II_F\RP_1(F) \arr \widetilde{\RP}_1(k) ,\ \ \ \ \ \  
\delta_0:\II_F\RP_1(F) \arr \II_k\widetilde{\RP}_1(k).
\]
By direct computation one sees that the composite 
\[
\II_A \RP_1(A)\arr \II_F\RP_1(F)\overset{\delta_\pi}{\arr} \widetilde{\RP}_1(k)
\]
is trivial. 

The main theorem (Theorem \ref{mainthm}) will follow from the following result (whose proof will occupy 
Subsections \ref{lgp} to \ref{obtain}): 

\begin{thm}\label{thm:4.2}
Let $A$ be a discrete valuation ring with field of fractions $F$ and sufficiently large residue field $k$. Then
\[
\begin{array}{c}
\II_A \RP_1(A)\half \arr \II_F\RP_1(F)
\half\overset{\delta_\pi}{\arr} \RP_1(k)\half \arr 0
\end{array}
\]
is an exact sequence of $\RR_A$-modules.
\end{thm}

\subsection{Characters and a local global principal}\label{lgp}

Here we review  a character-theoretic local-global principle for modules over 
elementary $2$-torsion abelian groups, developed in \cite[Section 3]{hutchinson2017}. This theory will be used
extensively in the proof of Theorem \ref{thm:4.2}. 

Let $G$ be a multiplicative  abelian group in which $g^2=1$ for all $g\in G$. Set $\RR:=\z[G]$.
The group $\widehat{G}:=\Hom(G, \{\pm1\})$ is called the group of characters of $G$.

For a character $\chi\in \widehat{G}$, let $\RR^{\chi}$ be the ideal of $\RR$ generated by the set
$\{g-\chi(g)|g\in G\}$. In fact $\RR^{\chi}$ is the kernel of the ring homomorphism $\rho_\chi:\RR\arr \z$ sending
$g\in G$ to $\chi(g)$ for any $g\in G$. Let $\RR_\chi:=\RR/\RR^{\chi}$, which is isomorphic to $\z$.

If $M$ is an $\RR$-module, we let  $M_{\chi}:=M/\RR^{\chi}M$. Furthermore, given $m\in M$ we will denote its image in $M_\chi$ by $m_\chi$.
 
In particular, if $\chi_0$ is the trivial character, then  $\RR^{\chi_0}$ is the augmentation ideal $\II_G$
and $M_{\chi_0}=M_G$.
Now we list some basic, but useful, properties of the ideals $\RR^\chi$ which we will use:
\par (i) For any  $\chi\in \widehat{G}$,  $\RR^{\chi}\half=(\RR^{\chi})^2\half$.
\par (ii) If $\chi_1,\chi_2\in \widehat{G}$ are distinct, then $\RR^{\chi_1}\half + \RR^{\chi_2}\half=\RR\half$.
\par (iii) If $\chi_1,\chi_2\in \widehat{G}$ are distinct and if $M$ is a $\RR$-module, then  
$\RR^{\chi_1}M\half \cap \RR^{\chi_2}M\half =\RR^{\chi_1}\RR^{\chi_2}M\half$.

\begin{lem}$($\cite[Corollaries 3.7, 3.8]{hutchinson2017}$)$\label{lem:exact}
The functors $M\mt \RR^{\chi}M$ and $M\mt M_{\chi}$ are exact on the category of $\RR\half$-modules.
In particular the functors $M\mt M_G$ and $M\mt \II_GM$ are exact on the category of $\RR\half$-modules.
\end{lem}

Here is the  character-theoretic local-global principle:

\begin{thm}$($\cite[Proposition 3.10]{hutchinson2017}$)$\label{local-global}
Let $f:M\arr N$ be a homomorphism of $\RR\half$-modules. For any $\chi\in \widehat{G}$, 
let $f_\chi$ be the induced homomorphism $M_\chi \arr N_\chi$.
\par {\rm (i)} $f$ is injective if and only if $f_\chi$ is injective for all $\chi\in \widehat{G}$.
\par {\rm (ii)} $f$ is surjective if and only if $f_\chi$ is surjective for all $\chi\in \widehat{G}$.
\par {\rm (iii)} $f$ is an isomorphism if and only if $f_\chi$ is an isomorphism for all $\chi\in \widehat{G}$.
\end{thm}
~\\
Here we fix some notations for later use. For $g\in G$, consider the orthogonal idempotents 
\[
e_+^g:=\frac{g+1}{2},  \ \ \ \ \ e_-^g:=\frac{g-1}{2},
\]
in $\RR\half$.
If $M$ is a $\RR\half$-module, then we have the decomposition 
\[
M=e_+^g M\oplus e_-^gM.
\]

For a ring $A$,  $\GG_A=\aa/(\aa)^2$ is a $2$-torsion group. If $M$ is a $\RR_A\half$-module we set
\[
M^+:=e_+^{-1}M=M/ e_-^{-1}M, \ \ \ \  M^-:=e_-^{-1}M=M/e_+^{-1} M.
\]

\subsection{The reduction and the specialization maps}

Let $\LL_{v}$ denote the $\RR_{F}$-submodule of $\widetilde{\RP}({F})$ generated by the elements $[{u}]$, $u\in U_{1,A}$. 

\begin{lem}\label{lem:comm}
There is a commutative diagram of $\RR_{A}$-modules with exact rows and columns
\[
\begin{tikzcd}
&&&0\ar[d]&\\
0\ar[r]&\LL_{A}\ar[r]\ar[d]&\widetilde{\RP}({A})\ar[r]\ar[d]&\widetilde{\RP}({k})\ar[r]\ar[d]&0\\
0\ar[r]&\LL_{v}\ar[r]&\widetilde{\RP}({F})\ar[r, "S_v"]&\Ind_k^{F}\widetilde{\RP}({k})\ar[r]\ar[d, "\delta_\pi"]&0\\
&&&\widetilde{\RP}({k})\ar[d]&\\
&&&0&\\
\end{tikzcd}
\]
Furthermore, the lower row in this diagram is a sequence of $\RR_{F}$-modules. 
\end{lem}
\begin{proof}
The top row is Lemma \ref{lem:slr}.  The lower row is \cite[Proposition 3.2]{hutchinson2019}. 
The column exact sequence is Lemma \ref{lem:decom}. The commutativity of the 
diagram and the exactness of the column are immediate from the definitions. 
\end{proof}

If we tensor this diagram with $\z\half$ and then multiply the rows by $e_+^{-1}\in \RR_{A}\half$,
then by Theorem \ref{hutchinson7} and  Proposition \ref{hutchinson6} we obtain the commutative 
diagram with exact rows and column
\[
\begin{tikzcd}
&&&0\ar[d]&\\
0\ar[r]&\LL_{A}^+\half\ar[r]\ar[d]&\RP_1(A)\half\ar[r]\ar[d]&\RP_1(k)\half\ar[r]\ar[d]&0\\
0\ar[r]&\LL_v^+\half\ar[r]&\RP_1(F)\half \ar[r, "S_v"]&\Ind_k^F\RP_1(k)\half\ar[r]\ar[d, "\delta_\pi"]&0\\
&&&\RP_1(k)\half\ar[d]&\\
&&&0&
\end{tikzcd}
\]
Now if we apply the exact functor $\II_{A}\cdot -$ to the top row and the exact functor $\II_{F}\cdot -$ to the 
lower row (see Lemma \ref{lem:exact}) we arrive at the following result.

\begin{cor}\label{cor:commd}
There is a commutative diagram of $\RR_{A}$-modules with exact rows and column
\[
\begin{tikzcd}
&&&0\ar[d]&\\
0\ar[r]&\II_{A}\LL_{A}^+\half\ar[r]\ar[d]&\II_{A}\RP_1(A)\half\ar[r]\ar[d]&\II_{k}\RP_1(k)\half\ar[r]\ar[d]&0\\
0\ar[r]&\II_F\LL_v^+\half\ar[r]&\II_F\RP_1(F)\half \ar[r, "S_v"]&\II_{F}\left(\Ind_k^{F}\RP_1(k)\half\right)\ar[r]\ar[d, "\delta_\pi"]&0\\
&&&\RP_1(k)\half\ar[d]&\\
&&&0&\\
\end{tikzcd}
\]
Furthermore, the lower row is an exact sequence of $\RR_{F}$-modules. 
\end{cor}
~\\
Given Corollary \ref{cor:commd}, Theorem \ref{thm:4.2} is equivalent to the following: 
\begin{itemize}
\item
The natural $\RR_{A}$-homomorphism  $\alpha:\II_{A}\LL_A^+\half\to\II_{F}\LL_v^+\half$ is surjective. 
 \end{itemize}
 By the character-theoretic local-global principle (Theorem \ref{local-global}), this in turn is equivalent to the statement:
 \begin{itemize}
 \item
The homomorphisms $\alpha_\chi:\Big(\II_{A}\LL_A^+\half\Big)_\chi\to\Big(\II_{F}\LL_v^+\half\Big)_\chi$ are 
surjective for all $\chi\in \widehat{\GG}_A$. 
\end{itemize}

We need the following lemma.

\begin{lem}\label{lem:M+}
Let $M$ be a $\RR_A\half$-module and $N$ a $\RR_F\half$-module. Then 
\par {\rm (i)} $(\II_{A}M)_{\chi_0}=0$, where $\chi_0$ is the trivial character on $\GG_A$.
\par {\rm (ii)} $(M^{+})_{\chi}=0=(N^{+})_{\chi}$ if $\chi \in\widehat{\GG}_A$ and $\chi(-1)=-1$.
\par {\rm (iii)} $(\II_AM^+)_{\chi}\simeq M_{\chi}$ and $(\II_FN^+)_{\chi}\simeq N_{\chi}$ if $\chi \in\widehat{\GG}_A$, 
$\chi\not=\chi_0$ and $\chi(-1)=1$. 
\end{lem}
\begin{proof}
(i) The triviality of $(\II_{A}M)_{\chi_0}$ follows from the fact that $\II_A\half=\II_A^2\half$.
\par (ii) First note that $(M^{+})_{\chi}=(e_+^{-1}M)_{\chi}=e_+^{-1}M/e_+^{-1}\RR_A^\chi M$.
If $\chi(-1)=-1$, then $e_+^{-1}\in \RR_A^\chi\half$. Since  $e_+^{-1}$ is an idempotent, we have 
$e_+^{-1}M=e_+^{-1}e_+^{-1}M\se e_+^{-1}\RR_A^\chi M$. Thus $(M^{+})_{\chi}=0$. Similarly $(N^{+})_{\chi}=0$.
\par (iii) We have the decomposition of $\RR_A$-modules
$\II_AM=e_+^{-1}\II_AM\oplus  e_-^{-1}\II_AM$. Let $\chi\not=\chi_0$ and $\chi(-1)=1$. Then clearly $e_-^{-1}\in \RR_A^\chi\half$.
Thus with a similar argument as in the proof of (ii), we have $(e_-^{-1}\II_AM)_\chi=0$. Therefore
$(\II_AM^+)_\chi=(\II_AM)_\chi$. Now it is easy to show that the natural map
$\beta: (\II_AM)_\chi\arr M_\chi$, induced by inclusion,  is an isomorphism. In fact since $\chi\neq \chi_0$,
there is $a\in \aa$ such that $\chi(a)=-1$. Thus $\Lan a\Ran m=(\lan a\ran-\chi(a))m-2m$ 
and hence $\beta(\Lan a\Ran m)=-2m$. This shows that $\beta$ is surjective. The injectivity of $\beta$ 
follows from the equality $\II_A\RR_A^\chi M=\II_A M\cap \RR_A^\chi M$.

The proof of the isomorphism $(\II_FN^+)_{\chi}\simeq N_{\chi}$ is similar. Just note that here we should
use the fact $\RR_A^\chi\II_F M=\RR_A^\chi M\cap \II_F M$ which follows from the fact that $\RR_A^\chi\half+\II_F\half=\RR_F\half$.
\end{proof}

\subsection{A decomposition lemma}

For any $\chi\in \widehat{\GG}_A$, we define two characters 
$\chi_+,\chi_-\in \widehat{\GG}_F$ (depending on our chosen uniformizer $\pi$) 
as follows:
For $u\in A^\times$, $r\in \z$,
\begin{eqnarray*}
&\chi_+(\pi^ru):= \chi(u)\\
&\chi_-(\pi^ru):=(-1)^r\chi(u).\\
\end{eqnarray*}
For example if $\chi_0 \in \widehat{\GG}_A$ is the trivial character, then ${\chi_0}_+$ is the trivial character
on $\GG_F$ and ${\chi_0}_-=\chi_v$, where $\chi_v\in \widehat{\GG}_F$ is the character associated to the valuation $v$
on $F$ given by $\chi_v(a)=(-1)^{v(a)}$.

\begin{lem} \label{lem:decom2}
Let $M$ be a $\RR_F\half$-module. Let $\chi\in  \widehat{\GG}_A$. Then the homomorphism
\[
{M}_{\chi}\to {M}_{\chi_+}\oplus{M}_{\chi_-}, \quad {m}_{\chi}\mapsto ({m}_{\chi_+},{m}_{\chi_-})
\]
is an isomorphism of $\RR_{A}$-modules.
\end{lem}
\begin{proof}
As a $\RR_{F}$-module we have the decomposition $M=e_+^\pi M\oplus e_-^{\pi}M$. Thus
\[
M_\chi=(e_+^\pi M)_\chi\oplus (e_-^{\pi}M)_\chi.
\]
We show that $(e_+^\pi M)_\chi\simeq M_{\chi_+}$. The proof of $(e_-^\pi M)_\chi\simeq M_{\chi_-}$ is similar. We have
\begin{align*}
(e_+^\pi M)_\chi& =(M/e_-^{\pi}M)_\chi=(M/e_-^{\pi}M)/\RR_A^\chi(M/e_-^{\pi}M) \\
&=(M/e_-^{\pi}M)/((\RR_A^\chi M+e_-^{\pi}M)/e_-^{\pi}M)\\
& \simeq M/(\RR_A^\chi M+e_-^{\pi}M).
\end{align*}
Since $M_{\chi_+}=M/\RR_F^{\chi_+}M$, we must  show that
$\RR_F^{\chi_+}M=\RR_A^\chi M+e_-^{\pi}M$. The ideal $\RR_F^{\chi_+}$
is generated by $\lan u\pi\ran -\chi(u)$ and $\lan u\ran -\chi(u)$, $u\in \aa$.
Clearly $\RR_A^\chi M+e_-^{\pi}M \se \RR_F^{\chi_+}M$.  Since
$\lan u\pi\ran -\chi(u)=2\lan u\ran e_-^\pi+(\lan u\ran -\chi(u))$,
we see that $\RR_F^{\chi_+}M\se \RR_A^\chi M+e_-^{\pi}M$. 
\end{proof}

Let $N:=\LL_v\half$.  By Lemma \ref{lem:decom2}, 
$(\II_FN^+)_{\chi_0}\simeq (\II_FN^+)_{{\chi_0}_+}\oplus (\II_FN^+)_{{\chi_0}_-} $, where $\chi_0\in \widehat{\GG}_A$
is the trivial character. Since ${\chi_0}_+$ is the trivial character on $\GG_F$,  $(\II_FN^+)_{{\chi_0}_+}=0$ 
(Lemma \ref{lem:M+}~(i)). Moreover ${\chi_0}_-=\chi_v$ and thus $(\II_FN^+)_{{\chi_0}_-}=(\II_FN^+)_{\chi_v}=\II_F(N^+)_{\chi_v}$.
By the proof of \cite[Theorem 3.7]{hutchinson2019}, $(\LL_v^+\half)_{\chi_v}=0$. Therefore 
\[
\begin{array}{c}
(\II_F\LL_v^+\half)_{\chi_0}=0.
\end{array}
\]
On the other hand $(I_A\LL_A^+\half)_{\chi_0}=0$. Thus by applying Lemma \ref{lem:M+} and Lemma \ref{lem:decom2}, 
Theorem \ref{thm:4.2} is equivalent to  the following statement:
\begin{itemize}
\item
 The homomorphisms
${\alpha}_{\chi}: (\LL_A\half)_\chi \mt (\LL_v\half)_\chi=(\LL_v\half)_{\chi_+} \oplus (\LL_v\half)_{\chi_-}$
are surjective for all $\chi\in \widehat{\GG}_A$ satisfying $\chi\not=\chi_0$ and $\chi(-1)=1$.
\end{itemize}

Since the $\RR_{F}$-module $\LL_{v}$ is generated by the elements $[{u}]$, $u\in U_{1,A}$, it follows that for any 
$\psi\in \widehat{\GG}_F$,  $(\LL_v\half)_\psi$ is generated as a $\z$-module by the elements $[u]_{\psi}$ for 
$u\in U_{1,A}$. Thus Theorem \ref{thm:4.2} is entirely equivalent to the following:

\begin{prp} \label{prop:4.2}
For all $u\in U_{1,A}$ and for all $\chi\in \widehat{\GG}_A$ satisfying $\chi\not=\chi_0$ and $\chi(-1)=1$, the elements 
$([u]_{\chi_+},0)$ and $(0,[u]_{\chi_-})$ lie in $ \im(\alpha_{\chi})$. 
\end{prp}

Proposition \ref{prop:4.2} (and hence Theorem \ref{thm:4.2}) follows from Lemmas \ref{lem:un} and \ref{lem:even1}  
and Corollaries \ref{cor:odd1} and \ref{cor:-1}  below.

\begin{lem}\label{lem:un} 
Let $\chi\in  \widehat{\GG}_A$. Suppose that $\chi|_{U_{n,A}}=1$ for some $n\geq 1$, where $U_{n,A}=1+\mmm_A^n$. Then
$(\LL_v\half)_{\chi_+} = (\LL_v\half)_{\chi_-}=0$.
\end{lem}

\begin{proof}
For $\delta\in \{ +,-\}$ the characters $\chi_\delta \in \widehat{\GG}_F$ have the property that $\chi_\delta|_{U_{n,A}}=1$. 
The lemma now follows from (the proof of) \cite[Theorem 3.11]{hutchinson2019}. (The reader may readily verify that 
the hypothesis $U_{n,A}\subset U_{1,A}^2$ in this theorem is only used in the proof to deduce that $\chi|_{U_{n,A}}=1$ for any 
given $\chi$.)
\end{proof}

\subsection{Obtainable elements}\label{obtain}

For the remainder of this section we fix $\chi\in \widehat{\GG}_A$ satisfying 
\begin{enumerate}
\item $\chi\not=\chi_0$,
\item $\chi(-1)=1$,
\item For all $n\geq 1$ there exists $u\in U_{n,A}$ with $\chi(u)=-1$.  
\end{enumerate}

Let us say that $u\in U_{1,A}$ is \emph{obtainable} if $([u]_{\chi_+},0)$ and $(0,[u]_{\chi_-})$ both lie in $ \im({\alpha}_{\chi})$. 
We must show that all $u\in U_{1,A}$ are obtainable. We begin with some  elementary observations: 

For $u\in U_{1,A}$ we set $\ell(u):= v(1-u)\in \N$. Thus $\ell(u)=n$ if and only if $u\in U_{n,A}\setminus U_{n+1,A}$.

\begin{lem}\label{lem:u1} 
{\rm (i)} For any $u\in U_{1,A}$ with $\ell(u)=n$ there exists $w\in U_{1,A}$ satisfying
 $w\equiv u\mod{U_{n+1,A}}$ and $\chi(w)=-\chi(u)$.
\par {\rm (ii)} For any $u\in U_{1,A}$ such that $\chi(u)=-1$ we have (for $\delta\in \{ +,-\}$)  
\[
\begin{array}{c}
[u]_{\chi_\delta}=\chi_\delta(1-u){\left( C_{F}\right)}_{\chi_\delta} \mbox{ in } (\LL_v\half)_{\chi_\delta}.
\end{array}
\]
\par {\rm (iii)} If $u\in U_{1,A}$ and if $\chi(u)=1$ and $\chi_{\delta}(1-u)=-1$ for some $\delta\in \{ +,-\}$, then $[u]_{\chi_\delta}=0$. 
\par {\rm (iv)} If $u\in U_{1,A}$, then
\[
\chi_-(1-u)=
\left\{
\begin{array}{ll}
\chi_+(1-u),& \ell(u) \mbox{ even}\\
-\chi_+(1-u),& \ell(u) \mbox{ odd.}\\
\end{array}
\right.
\]
\end{lem}

\begin{proof}
(i) Choose any $u\in U_{n,A}\setminus U_{n+1,A}$. Choose $z\in U_{n+1,A}$ with $\chi(z)=-1$ (see condition (3) listed above). 
Let $w=uz$. Then $w\equiv u\mod{U_{n+1,A}}$ and $\chi(w)=-\chi(u)$. 
\par (ii) By Corollary \ref{cor:key} we have
\[
\Lan{u}\Ran C_{F}=\lan{u-1}\ran\Lan{-u}\Ran[u] \mbox{ in } \widetilde{\RP}({F}).
\]
Thus for any $\psi\in \widehat{\GG}_F$ we have 
\[
(\psi(u)-1){\left(C_{F}\right)}_{\psi}=\psi(u-1)(\psi(-u)-1)[u]_{\psi} \mbox{ in } \widetilde{\RP}(F)_{\psi}.
\]
If $\psi(u)=-1$ and $\psi(-1)=1$, this gives
\[
-2{\left(C_{F}\right)}_{\psi}=-2\psi(1-u)[{u}]_{\psi}.
\] 
Now apply to $\psi=\chi_\delta$ and multiply both sides by $-1/2$.
\par (iii) By (ii) applied to $1-u$, we have $[1-u]_{\chi_\delta}=\left(C_F\right)_{\chi_\delta}$. However,
${\left(C_{F}\right)}_{\chi_\delta}=[{u}]_{\chi_\delta}+[{1-u}]_{\chi_\delta}$. 
\par (iv) We have $1-u=w\pi^{\ell(u)}$ for some $w\in A^\times$. Therefore
\[
\chi_-(1-u)=\chi(w)(-1)^{\ell(u)}=\chi_+(1-u)(-1)^{\ell(u)}.
\]
\end{proof}

We immediately deduce:

\begin{cor}\label{cor:u1}  
Let $u\in U_{1,A}$.
\par {\rm (i)} If $\chi(u)=-1$ and $\ell(u)$ is even, then
\[
\begin{array}{c}
[u]_{\chi}=\pm \left(C_F\right)_{\chi} \mbox{ in } \left(\LL_v\half\right)_{\chi}. 
\end{array}
\]
\par {\rm (ii)} If $\chi(u)=-1$ and $\ell(u)$ is odd, then
\[
\begin{array}{c}
[{u}]_{\chi}= \pm\left(\left(C_{F}\right)_{\chi_+},-{\left(C_{F}\right)}_{\chi_-}\right) \mbox{ in }
\left(\LL_v\half\right)_{\chi} =\left(\LL_v\half\right)_{\chi_+}\oplus \left(\LL_v\half\right)_{\chi_-}.
\end{array}
\]
\par {\rm (iii)} If $\chi(u)=1$ and $\ell(u)$ is even and $\chi_+(1-u)=-1$, then $[{u}]_{\chi}=0$.
\par {\rm (iv)} If $\chi(u)=1$ and $\ell(u)$ is odd, then
\[
[{u}]_{\chi}=\left([{u}]_{\chi_+},[{u}]_{\chi_-}\right)=
\left\{
\begin{array}{ll}
\left([{u}]_{\chi_+},0\right), & \chi_+(1-u)=1\\
\left(0,[{u}]_{\chi_-}\right), & \chi_+(1-u)=-1.
\end{array}
\right.
\]
\end{cor}

\begin{cor}\label{cor:even-1} 
 Let $u\in U_{1,A}$. If $\ell(u)$ is even and $\chi(u)=-1$, then $[{u}]_{\chi}\in\im({\alpha}_{\chi})$.
\end{cor}
\begin{proof} 
By Corollary \ref{cor:u1} (i), we only need to show that $\left(C_{F}\right)_{\chi}\in\im({\alpha}_{\chi})$.
Now if $x\in U_{1,A}$ with $\chi(x)=-1$, then $\Lan{x}\Ran C_{A}\in \LL_{A}$ and hence 
\[
{\alpha}_{\chi}\left( \Lan x\Ran C_{A}\right)=-2 {\left(C_{F}\right)}_{\chi}= {\left(C_{F}\right)}_{\chi}=
\left( {\left(C_{F}\right)}_{\chi_+},
{\left(C_{F}\right)}_{\chi_-} \right)\in \im({\alpha}_{\chi}). 
\]
\end{proof}

\begin{lem}\label{lem:odd} 
If $u\in U_{1,A}$ and if $\chi(u)=1$ and $\ell(u)$ is odd, then $[{u}]_{\chi}\in \im({\alpha}_{\chi})$.
\end{lem}
\begin{proof}
Let $n=\ell(u)$. Since $U_{n,A}/U_{n+1,A}\simeq k$ and since $|k|> 2$, there exists $w$ with 
$\ell(w)=\ell(uw)=n$. Furthermore, by 
Lemma \ref{lem:u1} (i), we can choose $w$ such that $\chi(w)=-1$. 
Now let 
\[
a:= \frac{1-w}{1-uw}, \mbox{ so that }w=\frac{1-a}{1-au}.
\]

We claim that $a\in \WW_{A}$: First note that $v(1-w)=\ell(w)=\ell(uw)=v(1-uw)$ and 
hence $a\in A^\times$. Now
\[
1-a = \frac{w-uw}{1-uw}=w\cdot\frac{1-u}{1-uw}
\]
which lies in $A^\times$ since $v(1-u)=\ell(u)=\ell(uw)=v(1-uw)$. This proves the claim. 

For $\delta\in \{ +,-\}$ we have 
\[
0={[{a}]}_{\chi_\delta}-{[{au}]}_{\chi_\delta}+\chi(a){[{u}]}_{\chi_\delta}-\chi(a)\chi(1-a){[{uw}]}_{\chi_\delta}
+\chi(1-a){[{w}]}_{\chi_\delta} 
\]
in ${\RP(F)\half}_{\chi_\delta}$. 

Since $\chi(uw)=\chi(w)=-1$ and since $1-w=a(1-uw)$, we have (by Lemma \ref{lem:u1}) that 
\[
\chi(a){[{uw}]}_{\chi_\delta}+{[{w}]}_{\chi_\delta}=\left(-\chi(a)\chi_\delta(1-uw)+\chi_\delta(1-w)\right){\left(C_{F}\right)}_{\chi_\delta}=0.
\]
Thus ${\left([{a}]-[{au}]\right)}_{\chi_\delta}=\chi(a)[{u}]_{\chi_\delta}$ for each $\delta\in \{ +,-\}$ and hence 
\[
{\alpha}_{\chi}\left( [{a}]-[{au}]\right)=\chi(a){[{u}]}_{\chi}.
\]
\end{proof}

Combining this with  Corollary \ref{cor:u1} (iv), we immediately deduce:

\begin{cor}\label{cor:odd1} 
If $u\in U_{1,A}$ satisfies $\ell(u)$ is odd and $\chi(u)=1$, then $u$ is obtainable.
\end{cor}

\begin{lem}\label{lem:bconst} 
The elements $\left( {\left({C_F}\right)}_{\chi_+},0\right)$ and $\left(0, {\left(C_{F}\right)}_{\chi_-}\right)$
lie in $\im({\alpha}_{\chi})$.
\end{lem}

\begin{proof} 
Let $u,w\in U_{1,A}$ satisfy: $\ell(u)$ is odd, $\ell(w)>\ell(u)$ is even and $\chi(u)=\chi(w)=-1$. Thus 
$\ell(uw)=\ell(u)$ is odd and $\chi(uw)=1$. Let $z:=(1-u)/(1-uw)$.

 We claim that $z\in U_{1,A}$ and $\ell(z)=\ell(w)-\ell(u)$ (and hence is odd): On the one hand, 
 $v(1-u)=\ell(u)=\ell(uw)=v(1-uw)$ implies that $z\in A^\times$. On the other hand, 
\[
1-z=\frac{u-uw}{1-uw}=u\cdot\frac{1-w}{1-uw}
\]
and $v(1-w)=\ell(w)>\ell(uw)=\ell(u)=v(1-uw)$. So $v(1-z)=\ell(w)-\ell(u)>0$ proving the claim. 

Let $\delta\in \{+,-\}$. In ${\LL_v\half}_{\chi_\delta}$ we have 
\[
0={[{u}]}_{\chi_\delta}-{[{uw}]}_{\chi_\delta}+\chi(u){[{w}]}_{\chi_\delta}+\chi_\delta(1-u){[{wz}]}_{\chi_\delta}+
\chi_\delta(1-u){[{z}]}_{\chi_\delta}
\]
(using $\chi_\delta(1-u^{-1})=\chi_\delta(u)\chi_\delta(1-u)=-\chi_\delta(1-u)$ since $\chi(u)=-1$).
\bigskip

There are two possibilities: $\chi(z)=-1$ or $\chi(z)=1$:
\bigskip

\textbf{Case 1:} $\chi(z)=-1$. Then we write 
\[
{[{uw}]}_{\chi_\delta}-\chi(u){[{w}]}_{\chi_\delta}-\chi_\delta(1-u){[{wz}]}_{\chi_\delta} =
{[{u}]}_{\chi_\delta}+\chi_\delta(1-u){[{z}]}_{\chi_\delta}. 
\]

Let
\[
X:= ([uw]-\chi(u)[w]-\chi(1-u)[wz])_\chi
\]
and observe that $X\in \im({\alpha}_{\chi})$ by Corollaries \ref{cor:even-1} and \ref{cor:odd1}, since $\ell(w)$ 
is even and $\chi(w)=-1$ and $\ell(uw),\ell(wz)$ are odd and $\chi(uw)=1=\chi(zw)$. Thus
\[
X= \left({[{u}]}_{\chi_+}+\chi_+(1-u){[{z}]}_{\chi_+},{[{u}]}_{\chi_-}+\chi_-(1-u){[{z}]}_{\chi_-}\right) \in \im({\alpha}_{\chi}).
\]
However since $\ell(u)$, $\ell(z)$ are odd and $\chi(u)=\chi(z)=-1$, by Lemma \ref{lem:u1} (ii) we have 
\begin{eqnarray*}
{[u]}_{\chi_\delta}+\chi_\delta(1-u){[z]}_{\chi_\delta}&=& \left(\chi_\delta(1-u)+
\chi_\delta(1-u)\chi_\delta(1-z)\right){\left(C_{F}\right)}_{\chi_\delta}\\
&=& \chi_\delta(1-u)\left(1+\chi_\delta(1-z)\right){\left(C_{F}\right)}_{\chi_\delta}.
\end{eqnarray*}
Again, since $\ell(z)$ is odd, $\chi_-(1-z)=-\chi_+(1-z)$ and hence $1+\chi_\delta(1-z)$ takes the 
value $0$ for one value of $\delta$ and $2$ for the other. It follows that $X$ is either 
$\pm\left( {\left(C_{F}\right)}_{\chi_+},0\right)$ or $\pm\left(0, {\left(C_{F}\right)}_{\chi_-}\right)$ 
and the Lemma follows.

\textbf{Case 2:} $\chi(z)=1$. In this case we have $\chi(wz)=-1$ and $\ell(wz)$ is odd. We write   
\[
{[uw]}_{\chi_\delta}-\chi(u){[w]}_{\chi_\delta}-\chi_\delta(1-u){[z]}_{\chi_\delta} =
{[u]}_{\chi_\delta}+\chi_\delta(1-u){[wz]}_{\chi_\delta} 
\]
and the argument of Case 1 applies by interchanging $z$ and $wz$.  
\end{proof}

Combining Lemma \ref{lem:bconst} with Corollary \ref{cor:u1} (i) and (ii) we immediately deduce:

\begin{cor}\label{cor:-1} 
If $u\in U_{1,A}$ and $\chi(u)=-1$, then $u$ is obtainable.
\end{cor}

We deal with the last remaining case:

\begin{lem}\label{lem:even1} 
If $u\in U_{1,A}$ satisfies $\chi(u)=1$ and $\ell(u)$ is even, then $u$ is obtainable.
\end{lem}
\begin{proof} Choose any $w$ with $\ell(w)=1$. Since $\ell(u)\geq 2$, $\ell(uw)=\ell(w)=1$.  Consider
$z:= (1-w)/(1-uw)$. Then $z\in U_{1,A}$ and $\ell(z)=\ell(u)-1$ which is odd. Similarly $uz\in U_{1,A}$ and 
$\ell(uz)=\ell(u)-1$.

Given $\delta\in \{+,-\}$, in ${\LL_v\half}_{\chi_\delta}$ we have relations 
\[
{[u]}_{\chi_\delta}=\pm{[w]}_{\chi_\delta}\pm{[uw]}_{\chi_\delta}\pm{[uz]}_{\chi_\delta}\pm{[z]}_{\chi_\delta}
\]
and since all the terms on the right are obtainable (since they satisfy $\ell(x)$ is odd), it follows that $u$ is 
obtainable also. 
\end{proof}

Thus we have completed the proof of Theorem \ref{thm:4.2}.  

\subsection{Proof of Theorem \ref{mainthm}}
We will need the following result from $K$-theory:
\begin{lem}\label{lem:kind}
Let $A$ be a discrete valuation ring with field of fractions $F$ and residue field $k$. Then the homomorphism 
$K_3^\ind(A)\to K_3^\ind(F)$ is surjective.
\end{lem}

\begin{proof} Consider the following commutative diagram with exact columns:
\[
\begin{tikzcd}
K_3^M(A)\ar[r]\ar[d]&K_3^M(F)\ar[r]\ar[d]&K_2^M(k)\ar[d, "\cong"]\ar[r]&0\\
K_3(A)\ar[r]\ar[d]&K_3(F)\ar[r]\ar[d]&K_2(k)\ar[r]&0\\
K_3^\ind(A)\ar[r]\ar[d]&K_3^\ind(F)\ar[d]&&\\
0&0&&\\
\end{tikzcd}
\]
The rows are exact by the localization exact sequences for $K$-theory and Milnor $K$-theory, together with the fact that the homomorphism 
$K_3^M(F)\to K_2^M(k)$ is surjective. (For a proof, see \cite[V.6.6.2]{weibel2013}.) The indicated right vertical arrow is an isomorphism by
Matsumoto's theorem. The lemma follows immediately. 
\end{proof}

Now let 
\[
\Delta_\pi:H_3(\SL_2(F),\z) \arr \widetilde{\RP}_1(k)
\]
be the composite $H_3(\SL_2(F),\z)\arr \RP_1(F)\overset{\delta_\pi}{\larr} \widetilde{\RP}_1(k)$.
By \cite[Lemma 7.1]{hutchinson2017}, the composite
\[
\begin{array}{c}
H_3(\SL_2(A),\z\half) \arr H_3(\SL_2(F),\z\half) \overset{\Delta_\pi}{\larr} \widetilde{\RP}_1(k)\half
\end{array}
\]
is the zero map.  Thus we have the commutative diagram
\[
\begin{tikzcd}
\II_A\RP_1(A)\bigg[ \frac{1}{2}\bigg]\ar[r]\ar[d]&\II_F\RP_1(F)\bigg[\frac{1}{2}\bigg]\ar[r,"{\delta_\pi}"]
\ar[d]&\RP_1(k)\bigg[ \frac{1}{2}\bigg]\ar[r]\ar[d, "="]&0\\
H_3(\SL_2(A),\z\bigg[ \frac{1}{2}\bigg])\ar[r]\ar[d]&H_3(\SL_2(F),\z\bigg[ \frac{1}{2}\bigg])\ar[r, "{\Delta_\pi}"]\ar[d]&
\RP_1(k)\bigg[\frac{1}{2}\bigg]\ar[r]&0\\
K_3^\ind(A)\ar[r, two heads]\ar[d]&K_3^\ind(F)\ar[d]&&\\
0&0&&\\
\end{tikzcd}
\]
in which the columns are exact (by Proposition \ref{hutchinson4}). The top row is exact (by Theorem \ref{thm:4.2}) and the 
middle row is a complex. Since the bottom horizontal arrow is surjective by Lemma \ref{lem:kind}, a straightforward diagram 
chase establishes the exactness of the middle row. Thus Theorem \ref{mainthm} is proved.

\section{The Mayer-Vietoris exact homology sequence of  \texorpdfstring{$\SL_2(A)$}{Lg}}\label{mv}

As in Section \ref{sec:RSCG}, let $A$ be a discrete valuation ring with maximal ideal $\mmm_A$. Let $k$, $F$, $v=v_A$ 
be the residue field, field of fractions and discrete valuation of $A$. Let $\pi$ be a fixed choice of uniformizer.

In this section we review the relevant facts about the tree associated to the discrete valuation $v$, the resulting decomposition of $\SL_2(F)$ as an amalgamated product and the Mayer-Vietoris exact sequence in homology. We show that Mayer-Vietoris is naturally a sequence of $\RR_F$-modules and explicitly describe this action on the relevant terms of the sequence.    

\subsection{Rank two lattices and the associated tree}
Let $F^2$ be the usual $F$-vector space of dimension $2$.
Any rank two free $A$-submodule of $F^2$ is called a {\it lattice}. 
If $x\in \ff$ and $L$ is a lattice of $F^2$, $Lx$ is also a lattice of $F^2$. 
Thus $\ff$ acts on the set of lattices of $F^2$. 

Two lattices $L$ and $L'$ are said to be
equivalent if there is an element $x\in \ff$ such that $L=L'x$. In other words, two lattices are equivalent if
their orbits under the action of $\ff$ are the same. The set of equivalence classes of lattices of
 $F^2$  will be denoted by $V$. The following review  follows the standard account in  \cite{serre1980}.

Let $L$ and $L'$ be two lattices of $F^2$. By the Invariant Factor Theorem, there is an 
$A$-basis $\{{\bf e}, {\bf f}\}$ of $L$ and integers $a, b$ such that $\{{\bf e}\pi^a, {\bf f}\pi^b\}$ is an
$A$-basis for $L'$. The integers $a,b$ does not depend on the choice of bases
for $L, L'$ and $|a-b|$ depends only on the equivalence classes $\Lambda, \Lambda'$
of $L$ and $L'$. Furthermore $L'\subseteq L$ if and only if $a,b \in \z_{\geq 0}$, in which case
\[
L/L'\simeq A/\pi^aA\oplus A/\pi^bA.
\]

The distance between two classes $\Lambda, \Lambda'\in V$ is defined as
\[
d(\Lambda, \Lambda'):=|a-b|.
\]
If $L$ is a lattice, each class $\Lambda' \in V$ has exactly one representative 
$L'$, satisfying $L'\subseteq L$ and $L'\nsubseteq L\pi$ (or equivalently $L'\subseteq L$
and $L'$ is maximal in $\Lambda'$ with this property). In this case we have $L/L'\simeq A/\pi^n A$,
where $n=d(\Lambda, \Lambda')$. In particular
\par (i) $d(\Lambda, \Lambda')=0$ if and only if $\Lambda=\Lambda'$,
\par (ii) $d(\Lambda, \Lambda')=1$ if and only if there are representatives $L'\subseteq L$
of $\Lambda'$ and $\Lambda$ such that $L/L'\simeq k$.

Two elements $\Lambda, \Lambda'$ of $V$ are siad to be  \textit{adjacent} if $d(\Lambda, \Lambda')=1$. In
this way we can define a combinatorial graph structure on $V$. We denote this graph, whose vertices are the  elements of $V$, by $T_V$. The structure of this graph is known:

\begin{thm}$($\cite[pp. 70-72]{serre1980}$)$\label{tree}
The graph $T_V$ is a tree. Moreover the edges with origin $\Lambda$ correspond
bijectively to the points of $\Pp^1(k)$.
\end{thm}

\subsection{Amalgamated product decomposition}

The action of $G:=\SL_2(F)$ on $V$ has two orbits say with representatives
$\Lambda_0$ and $\Lambda_1$ which are the equivalence classes of $L_0:=A\oplus A$ and 
$L_1:=A\oplus \pi A$, respectively. 

The stabilizer of the vertices $\Lambda_0$, $\Lambda_1$ 
and the edge $(\Lambda_1,\Lambda_0)$ are
\begin{align*}
&{G}_{\Lambda_0}={G}_{L_0}=\{g\in {G}:
g\Lambda_0=\Lambda_0\}=\SL_2(A),\\
& {G}_{\Lambda_1}={G}_{L_1}=\{g\in {G}:
g\Lambda_1=\Lambda_1\}=\SL_2(A)^{g_\pi}
=:g_\pi\SL_2(A)g_\pi^{-1},\\
& {G}_{(\Lambda_1,\Lambda_0)}={G}_{(L_1,L_0)}=\{g \in{G} : 
g(\Lambda_1,\Lambda_0)=(\Lambda_1,\Lambda_0)\}
={G}_{\Lambda_1}\cap {G}_{\Lambda_0}=:\Gamma_0(\mmm_A),
\end{align*}
where $g_\pi={\mtxx 0 {-1} {\pi} 0}$ \cite[Section 1.3, Chap. II]{serre1980}. Observe that
\[
g_\pi\SL_2(A)g_\pi^{-1}=\bigg\{ {\mtxx a {b\pi} {\pi^{-1}c} d}|{\mtxx a b c d}\in \SL_2(A)\bigg\},
\]
\[
\Gamma_0(\mmm_A)=\bigg\{ {\mtxx a b c d}\in \SL_2(A)| c\in \mmm_A\bigg\}.
\]
Note that for $h_\pi:={\mtxx {\pi} 0 0 1}$, we have  ${G}_{\Lambda_1}=
h_\pi^{-1}\SL_2(A)h_\pi=\SL_2(A)^{h_\pi^{-1}}$.
Serre's theory of Trees \cite[Chap. II]{serre1980} 
allows us to deduce:

\begin{thm}(Ihara)\label{amalgamated} The group $\SL_2(F)$ is the sum of the 
subgroups $\SL_2(A)$ and $\SL_2(A)^{g_\pi}$ amalgamated along their common 
intersection $\Gamma_0(\mmm_A)$:
\[
\SL_2(F)= \SL_2(A)\ast_{\Gamma_0(\mmm_A)} \SL_2(A)^{g_\pi}.
\]
\end{thm}

\subsection{The action of \texorpdfstring{$\GL_2(F)$}{Lg} on the singular complex of the tree}

Let $G$ denote $\SL_2(F)$, as above, and let $\widetilde{G}=\GL_2(F)$. 
The group $\widetilde{G} $ acts transitively on the set of vertices, $V$, of the tree $T_V$. 
The matrix $g_\pi={\mtxx 0 {-1} {\pi} 0}\in \widetilde{G}$ transforms $L_0$ into $L_1$ and 
$L_1$ into $L_0\pi$. Therefore it transforms $\Lambda_0$ to  $\Lambda_1$,
 $\Lambda_1$ to  $\Lambda_0$ and the edge
$(\Lambda_0,\Lambda_1)$ into its opposite $(\Lambda_1,\Lambda_0)$. 
The collection 
\[
E^+:=\{(g\Lambda_1,g\Lambda_0)|g\in G\}
\]
gives a set of oriented edges of $T_V$. Since $T_V$ is contractible, its singular complex
\[
0 \arr \z[E^+] \arr \z[V] \arr \z \arr 0
\]
is an exact sequence of $\z[G]$-modules. We let $\widetilde{G}$ act on $\z[E^+]$ as follows:
\[
g.(\Lambda,\Lambda'):=
\begin{cases}
+(g\Lambda,g\Lambda') & \text{if $(g\Lambda,g\Lambda')\in E^+$}\\
-(g\Lambda,g\Lambda') & \text{if $(g\Lambda',g\Lambda)\in E^+$}.
\end{cases}
\]
For $g\in \widetilde{G}$, let $\epsilon(g)\in\{0,1\}\subset \z$ be defined by 
\[
v\circ \det(g)\equiv \epsilon(g)\!\!\!\!\pmod 2.
\]
In fact we have

\begin{lem}
If $\Lambda \in V$ and $g\in \widetilde{G}=\GL_2(F)$, then
$v\circ \det(g)\equiv d(\Lambda,g\Lambda) \!\pmod 2$.
\end{lem}
\begin{proof}
See \cite[Corollary, p. 75]{serre1980}.
\end{proof}

Thus if $e \in E^+$, $g \in \widetilde{G}$ we have
$
g.e=
\begin{cases}
+ge & \text{if $\epsilon(g)=0$}\\
-ge & \text{if $\epsilon(g)=1$}
\end{cases}
$.
With this action, the above sequence is a sequence of $\z[\widetilde{G}]$-modules.
Thus the corresponding long exact homology sequence
\[
\cdots\arr H_i(G,\z[E^+]) \arr H_i(G,\z[V])\arr H_i(G,\z)\overset{\delta}{\larr}
H_{i-1}(G, \z[E^+])\arr \cdots
\]
is a sequence of $\widetilde{G}$-modules, letting $\widetilde{G}$ act on $G$ 
by conjugation. Since the restriction of this action to $G$ is trivial, it follows that 
the homology sequence is a sequence of $\ff$-modules, because 
$\widetilde{G}/G\overset{\det}{\simeq} \ff$. Since furthermore diagonal matrices 
act trivially, $(\ff)^2$ acts trivially on the sequence. Therefore the above
long exact sequence is a sequence of $\RR_F$-modules.

\subsection{The Mayer-Vietoris exact sequence}\label{MV0}

Let $G_0=\SL_2(A)$, $G_1:=\SL_2(A)^{g_\pi}$ and  $\Gamma=\Gamma_0(\mmm_A)$.
The group $G_i$ is the stabilizer in $G$ of the $\Lambda_i$, for $i=0,1$ and we have bijective correspondences
\[
G/G_0\sqcup G/G_1 \longleftrightarrow V, \ \ \ \  gG_i \leftrightarrow g \Lambda _i, \  \ i=0,1
\]
and 
\[
G/\Gamma \longleftrightarrow E^+, \ \ \ \  \ g\Gamma \leftrightarrow (g\Lambda _1, g\Lambda_0).
\]
With these identifications, the singular complex of $T_V$ has the form
\[
0 \arr \z[G/\Gamma] \arr \z[G/G_0] \oplus \z[G/G_1] \arr \z \arr 0.
\]
The resulting long exact sequence has the form
\[
\cdots\arr H_i(G, \z[G/\Gamma]) \arr H_i(G, \z[G/G_0])\oplus H_i(G, \z[G/G_1]) \arr H_i(G,\z)  
\]
\[
\overset{\delta}{\larr} H_{i-1}(G, \z[G/\Gamma])\arr \cdots.
\]
Hence by the Shapiro Lemma, we have the Mayer-Vietoris homology exact sequence
\[
\begin{CD}
\cdots\larr H_i(\Gamma,\z) @>{(i_\ast,-i_\ast')}>> H_i(G_0, \z)\oplus H_i(G_1, \z) 
@>{j_\ast+j_\ast'}>>  H_i(G,\z) 
\overset{\delta}{\larr} H_{i-1}(\Gamma,\z) \larr \cdots
\end{CD}
\]
which is an exact sequence of $\RR_F$-modules.  Here $i:\Gamma \arr G_0$, 
$i':\Gamma\arr G_1$, 
$j:G_0 \arr G$ and $j':G_1\arr G$ are the usual inclusion maps. Thus we have

\begin{thm}\label{MV1}
For any discrete valuation ring $A$ we have
the Mayer-Vietoris exact sequence of $\RR_F$-modules
\[
\begin{CD}
\cdots \arr H_i(\SL_2(A),\z)\oplus H_i(\SL_2(A),\z) & \overset{\beta}{\larr} &
H_i(\SL_2(F),\z) \overset{\delta}{\larr} H_i(\Gamma_0(\mmm_A),\z)    \overset{\alpha}{\larr}
\end{CD}
\]
\[
\begin{CD}
 H_{i-1}(\SL_2(A),\z)\oplus H_{i-1}(\SL_2(A),\z) & \overset{\beta}{\larr} &
 H_{i-1}(\SL_2(F),\z)\larr \cdots,
\end{CD}
\]
where $\beta(z,z')=j_\ast(z)+\lan \pi\ran j_\ast(z')$ and $\alpha(x)=(i_\ast(x), -i_\ast (\lan \pi\ran\bullet x))$. Here
$i:\Gamma_0(\mmm_A)\arr \SL_2(A)$ and $j:\SL_2(A) \arr \SL_2(F)$  are the usual inclusions and $\lan \pi\ran\bullet x$
is the action of $\lan \pi\ran$ on $H_n(\Gamma,\z)$ induced by conjugation.
\end{thm}
\begin{proof}
Based on what we have discussed in the above, to finish the proof we only need to describe the maps $\alpha$ and $\beta$.
If  $C_\pi:G\arr G$, is the conjugation map $C_\pi(g)= g_\pi g g_\pi^{-1}$, then from the
commutative diagram
\[
\begin{tikzcd}
H_i(\Gamma,\z) \ar[r, " i_\ast "] \ar[d, "{(C_\pi |_\Gamma)}_\ast"] &H_i(G_0,\z) \ar[r, " j_\ast "] 
\ar[d, "{(C_\pi |_{G_0})}_\ast"] & H_i(G,\z) \ar[d, "{C_\pi}_\ast"]\\
H_i(\Gamma,\z) \ar[r, " i_\ast ' "] &H_i(G_1,\z) \ar[r, "j\ast' "]                                   &H_i(G,\z)
\end{tikzcd}
\]
we see that $\lan \pi\ran j_\ast (z)=j_\ast'({C_\pi}_\ast (z))$ and $i'(\lan \pi \ran\bullet x)={C_\pi}_\ast\circ i_\ast(x)$.  
If we replace $x$ with $\lan \pi \ran\bullet x$, then we have $i_\ast'(x)={C_\pi}_\ast\circ i_\ast(\lan \pi \ran\bullet x)$.
This completes the proof.
\end{proof}

\subsection{Explicit action of \texorpdfstring{$\GL_2(F)$}{Lg} on the singular complex of the tree}

Now we make explicit the action of $\widetilde{G}$ on (basis elements of)
$\z[G/G_0] \oplus \z[G/G_1]$ and $\z[G/\Gamma]$. For any $g\in \widetilde{G}$, let
\[
\det(g)=\pi^{2s(g)+\epsilon(g)}u_g,
\]
where $s(g)\in \z$, $\epsilon(g)\in \{0,1\}$ and $u_g\in \aa$.  Then we can write 
\[
g= {\mtxx \pi  0 0 \pi}^{s(g)}R(g){\mtxx {u_g} 0 0 1}g_\pi^{\epsilon(g)} ,
\]
where $R(g)\in G$.
As we have seen before $g_\pi.\Lambda _0=\Lambda_1$ and $g_\pi.\Lambda _1=\Lambda_0$.
So using the identification of the sets  $G/G_0 \sqcup G/G_1 \longleftrightarrow V$, we have 
$g_\pi G_0=G_1$ and $g_\pi G_1=G_0$. So $gG_i=R(g)G_{i+\epsilon(g)}$. More generally
\[
g.(xG_i)=R(gx).G_{i+\epsilon(gx)}.
\]
In particular for $\tilde{u}:= {\mtxx u 0 0 1}$, where $u\in \aa$,  we have $\tilde{u}G_i=G_i$, $i=0,1$, since 
$R(\tilde{u})=I_2$ and $\epsilon(\tilde{u})=0$.
Moreover, $g_\pi(\Lambda_1,\Lambda_0)=(\Lambda_0,\Lambda_1)$ and thus 
$g_\pi.\Gamma=-\Gamma$. Using this, we have  $g.\Gamma=(-1)^{\epsilon(g)}R(g).\Gamma$ 
and so, more generally, 
\begin{eqnarray}\label{eq:gamma}
g.(x\Gamma)=(-1)^{\epsilon(gx)}R(gx).\Gamma.
\end{eqnarray}
So with these explicitly defined actions, the exact sequence
\[
0 \arr \z[G/\Gamma] \arr \z[G/G_0] \oplus \z[G/G_1] \arr \z \arr 0,
\]
with the maps $x\Gamma \mapsto (xG_0,-xG_1)$ and $(nxG_0, myG_1)\mapsto n+m$,
is a sequence of $\widetilde{G}$-modules. Hence the resulting Mayer-Vietoris exact sequence, 
obtained by applying the functor $H_\bullet(G, -)$ is a sequence of $\RR_F$-modules.

\subsection{The action of \texorpdfstring{$\RR_F$}{Lg} on \texorpdfstring{$H_\bullet(G_0,\z)\oplus H_\bullet(G_1,\z)$}{Lg} }

Now we study the action of $\GG_F$ on the terms of the Mayer-Vietoris exact sequence.

\begin{prp}\label{switch}
Let $M=\z[G/G_0]\oplus \z[G/G_1]$ and let $C_\pi':G_1\arr G_0$ be given by $g\mapsto g_\pi^{-1}gg_\pi$.
Then with the isomorphism ${C_\pi'}_\ast:H_\bullet(G_1,\z) \arr H_\bullet(G_0,\z)$,
the action of $\lan \pi \ran\in \RR_F$ on $H_\bullet(G,M)\simeq H_\bullet(G_0,\z)\oplus 
H_\bullet(G_0,\z)$ is to interchange the two factors.
\end{prp}
\begin{proof}
First let study the action of $\GG_F$ on $H_\bullet(G,M)$ on the level of chain complex. 
Let $P_\bullet \arr \z$ be a right resolution of $\z$ over $\widetilde{G}$. 
Let $a\in \ff$ and choose $g\in \widetilde{G}$ with $\det(g)=a$. Then the multiplication 
by $\lan a\ran$ is given at the level of chains by the morphism
\[
P_\bullet \otimes_G M\arr P_\bullet \otimes_G M, \ \ \ \ \  x\otimes m\mapsto  xg^{-1}\otimes gm.
\]
It follows that $u\in \aa$ acts in the standard way, i.e. via conjugation by
$\tilde{u}$,  on each of the factors $H_\bullet(G_0,\z)$ and $H_\bullet(G_1,\z)$. 
Furthermore the isomorphism
\[
H_\bullet(G_1,\z) \overset{\simeq}{\larr} H_\bullet(G_0,\z)
\]
induced by ${C_\pi'}_\bullet :P_\bullet \otimes_{G_1} \z\arr 
P_\bullet \otimes_{G_0}\z $,  $ x\otimes 1\mapsto  xg_\pi\otimes 1$
is an $\aa$-module isomorphism, since $\det (g_\pi\tilde{u}g_\pi^{-1})=\det(\tilde{u})=u$.

Now since $\det(g_\pi)=\pi$, we look at the action of $g_\pi$ on $P_\bullet \otimes_G M$. 
Let $i_1$ and $i_2$ be the
inclusions $\z[G/G_0]\arr M$ and $\z[G/G_1]\arr M$. The diagrams 
\[
\begin{tikzcd}
P_\bullet \otimes_G M \ar[r, "{C_\pi'}_\bullet"]                      & P_\bullet \otimes_G M                                                             
&                                                                                &\\
P_\bullet \otimes_{G} \z[G/G_0] \ar[r]\ar[u, "{i_1}_\bullet"]
& P_\bullet \otimes_G \z[G/G_1]\ar[u, "{i_2}_\bullet"]             
&  x\otimes G_0 \ar[r, mapsto]                              & xg_\pi^{-1}\otimes G_1\\
& P_\bullet \otimes_{G_1} \z \ar[u,"\simeq"]\ar[d,"\simeq"]    &                                                                                 
& xg_\pi^{-1}\otimes 1\ar[u, mapsto]\ar[d,mapsto]\\
P_\bullet \otimes_{G_0} \z \ar[r] \ar[uu]                                 & P_\bullet \otimes_{G_0} \z                                                       
&  x\otimes 1 \ar[r, mapsto, "="] \ar[uu, mapsto] & x\otimes 1
\end{tikzcd}
\]
and 
\[
\begin{tikzcd}
P_\bullet \otimes_G M \ar[r, "{C_\pi'}_\bullet"]                       & P_\bullet \otimes_G M                                                             
&                                                                                             &\\
P_\bullet \otimes_{G} \z[G/G_1] \ar[r]\ar[u, "{i_2}_\bullet"]   
& P_\bullet \otimes_G \z[G/G_0]\ar[u, "{i_1}_\bullet"]            
&  xg_\pi^{-1}\otimes G_1 \ar[r, mapsto]                          & x\otimes G_0\\
P_\bullet \otimes_{G_1} \z \ar[u, "\simeq"]\ar[d,  "\simeq"]   &   
&    xg_\pi^{-1}\otimes 1\ar[u, mapsto]\ar[d, mapsto]     & \\
P_\bullet \otimes_{G_0} \z \ar[r]            & P_\bullet \otimes_{G_0} \z  \ar[uu]                                         
&  x\otimes 1 \ar[r, mapsto, "="]     & x\otimes 1\ar[uu, mapsto]
\end{tikzcd}
\]
commute. This implies the claim.
\end{proof}

\begin{cor}
For any $n$, there is an isomorphism of $\RR_F$-modules  
\[
H_n(G,\z[G/G_0]\oplus \z[G/G_1])\simeq \Ind_{A}^{F} H_n(G_0,\z)=\RR_F \otimes_{\RR_A} H_n(G_0,\z).
\]
 In particular the Mayer-Vietoris exact 
sequence has the following form as an sequence of $\RR_F$-modules
\[
\begin{CD}
\cdots \arr \Ind_A^F H_i(\SL_2(A),\z) & \overset{\beta}{\arr} & H_i(\SL_2(F),\z)\! \overset{\delta}{\arr} H_i(\Gamma_0(\mmm_A),\z)    
\overset{\alpha}{\arr}\Ind_A^F H_{i-1}(\SL_2(A),\z) & \overset{\beta}{\arr} \cdots
\end{CD}
\]
\end{cor}
\smallskip

\subsection{The two actions of \texorpdfstring{$\RR_F$}{Lg} on  \texorpdfstring{$H_n(\Gamma, \z)$}{Lg}}

We begin by noting that there is a natural action of $\RR_F$ on $H_\bullet(\Gamma,\z)$ compatible with the inclusion
$\Gamma \harr G$: Let
\[
\widetilde{\Gamma}:=\ff\Gamma\widetilde{U}\lan g_\pi\ran,
\]
where $\widetilde{U}:=\bigg\{ {\mtxx u 0 0 1}| u \in \aa\bigg\}$. Then $\widetilde{\Gamma}$ is a subgroup of 
$\widetilde{G}$ (observe that $g_\pi^2=-\pi I_2$). Note that if $g\in \widetilde{\Gamma}$, then
\[
g={\mtxx \pi  0 0 \pi}^{s(g)} R(g) {\mtxx {u_g^{-1}} 0 0 1}g_\pi^{\epsilon(g)},
\]
where now $R(g)\in \Gamma$. There is a map of group extensions
\[
\begin{tikzcd}
1 \ar[r] &\Gamma \ar[r] \ar[d] &\widetilde{\Gamma}  \ar[r, "\det"] \ar[d] & \ff \ar[r] \ar[d, "="] & 1\\
1 \ar[r] & G \ar[r]  & \widetilde{G}  \ar[r, "\det"] & \ff \ar[r] & 1.
\end{tikzcd}
\]
Thus $\z[\ff]$ acts on $H_\bullet(\Gamma,\z)$ in such a way that the map $H_\bullet(\Gamma,\z)\to H_\bullet(G,\z)$ is a map of 
$\z[\ff]$-modules. However, clearly $(\ff)^2$ acts trivially also on the first term, so that this is a map of $\RR_F$-modules. 

We denote this action of $\RR_F$ on $H_n(\Gamma,\z)$  by $\lan a\ran \bullet x$ (where $\lan a \ran\in \GG_F$ and 
$x\in H_n(\Gamma,\z)$). We will refer to it as the \emph{natural action} of $\RR_F$ on $H_\bullet(\Gamma,\z)$.

However, this action of $\RR_F$ on $H_n(\Gamma,\z)$ \emph{is not the same as the $\RR_F$-module action associated to 
the Mayer-Vietoris sequence }(as in Theorem \ref{MV1}). We now describe the relationship between these two actions, the 
natural action and the Mayer-Vietoris action:

Given an $\RR_F$-module $M$ and a character $\chi \in \widehat{\GG}_F$, we can define
the $\chi$-twisted $\RR_F$-module $M_{(\chi)}$ by
\[
\lan a \ran \ast_{\chi} m :=\chi(a)\lan a\ran m.
\]

Recall that the discrete valuation $v$ induces a character $\chi_v\in \widehat{\GG}_F$ defined by $\chi_v(a):=(-1)^{v(a)}$.

\begin{thm}\label{twisted}
The action of $\RR_F$ on $H_\bullet(\Gamma,\z)$ which is induced from the Mayer-Vietoris exact sequence
is the $\chi_v$-twist of the natural action, i.e. 
$\lan a\ran x=(-1)^{v(a)}\lan a\ran \bullet x$ for any $\lan a\ran \in \GG_F$ and $x\in H_n(\Gamma,\z)$.
\end{thm}
\begin{proof}
We begin by recalling the $\RR_F$-action associated to the Mayer-Vietoris sequence on $H_n(\Gamma,\z)$:
Let $C_\bullet \arr \z$ be a right projective resolution of $\z$ over $\z[\tilde{G}]$.  We can use the complex
 $C_\bullet\otimes_G\z[G/\Gamma]$ to calculate $H_\bullet(G,\z[G/\Gamma])\cong H_\bullet(\Gamma,\z)$. Let $g\in \tilde{G}$. 
Then the action of the class of $\mathrm{det}(g)$ on $H_\bullet(\Gamma,\z)$ is induced from the map
\begin{eqnarray*}
C_\bullet\otimes_G\z[G/\Gamma]&\to &C_\bullet\otimes_G\z[G/\Gamma],\\
 z \otimes x\Gamma& \mapsto &zg^{-1}\otimes g\cdot(x\Gamma)=
(-1)^{\epsilon(gx)}zg^{-1}\otimes R(gx)\Gamma
\end{eqnarray*}
using (\ref{eq:gamma})  above.

Now let $\lan a \ran \in \GG_F$ be represented by $u\pi^\epsilon \in \ff$, where $u\in \aa$ and $\epsilon\in \{0,1\}$.
Let $g=g(a):=\tilde{u}g_\pi^\epsilon \in \widetilde{\Gamma}$. Thus $R(g)=1$ and $g\Gamma=(-1)^\epsilon\Gamma$ and 
hence the Mayer-Vietoris action of $\lan a\ran$ on $H_\bullet(\Gamma,\z)$ is induced by the map 
$z\otimes \Gamma \mapsto (-1)^\epsilon zg^{-1}\otimes \Gamma$. 

The map 
\[
C_\bullet \otimes_\Gamma \z\to C_\bullet\otimes_G\z[G/\Gamma], z\otimes 1\mapsto z\otimes \Gamma
\]
induces an isomorphism on homology by Shapiro's lemma. We conclude the proof of the theorem by noting that for 
any $g\in\tilde{\Gamma}$, the map induced on $H_\bullet(\Gamma,\z)$ by conjugation by $g$ is described at the 
level of the complex $C_\bullet\otimes_\Gamma\z$ by $z\otimes 1\mapsto zg^{-1}\otimes 1$. 
\end{proof}

\begin{exa}
Here we show that the action of $\RR_F$ on $H_1(\Gamma,\z)$ is trivial, provided that $k$ has at least
$4$ elements. 

The natural map
$\theta: \Gamma \arr T_k$ given by ${\mtxx a b c d}\mapsto {\mtxx {\bar a} 0 0 {{\bar a}^{-1}}}$ induces the isomorphism
\[
H_1(\Gamma,\z)\simeq H_1(T_k,\z)\simeq k^\times,
\]
(see  the proof Theorem \ref{exact1}  below).
Since $H_1(\Gamma,\z)=\Gamma/\Gamma'$, it follows that if the diagonal of two matrices in 
$\Gamma$ are the same, then they represent the same element of $H_1(\Gamma,\z)$.

To study the action of $\RR_F$, it is enough to study the action of the elements 
$\lan u\ran$, $u\in \aa$ and $\lan \pi \ran$.
If $x\in H_1(\Gamma,\z)=\Gamma/\Gamma'$ is represented by ${\mtxx a b c d}$, 
then by Theorem \ref{twisted}  we have
\[
\lan u\ran.x= \bigg({\mtxx u 0 0 1} {\mtxx a b c d} {\mtxx {u^{-1}} 0 0 1}\bigg)^{(-1)^{v(u)}} \Gamma'
={\mtxx a {ub} {u^{-1}c} d}\Gamma'={\mtxx a b c d}\Gamma'=x,
\]
and 
\[
\lan \pi\ran.x=\bigg( g_\pi {\mtxx a b c d} g_\pi^{-1}\bigg)^{(-1)^{v(\pi)}} \Gamma'={\mtxx d {-c\pi^{-1}} {-\pi b} a}^{-1}\Gamma'=
{\mtxx d {-b} {-c} a}^{-1}\Gamma'=x.
\]
Thus the action of $\RR_F$ on $H_1(\Gamma,\z)$ is trivial.
\end{exa}

\section{The connecting homomorphism and the structure of  \texorpdfstring{$H_2(\Gamma_0(\mmm_A),\z\half)$}{Lg}}\label{connecting}

In this section we will give an explicit formula for the connecting homomorphism
\[
\begin{array}{c}
\delta:H_3(\SL_2(F),\z\half)=H_3(G,\z)\arr H_2(\Gamma,\z)=H_2(\Gamma_0(\mmm_A),\z\half)
\end{array}
\]
and a description of its image. Moreover we describe the kernel and the cokernel of the natural map
\[
\begin{array}{c}
H_2(\Gamma_0(\mmm_A),\z\half)\to H_2(\SL_2(A),\z\half). 
\end{array}
\]

\subsection{The image of \texorpdfstring{$\delta$}{Lg}}

We will begin by using Theorem \ref{mainthm} to identify $\coker(\beta)\cong\im(\delta)$ as an $\RR_F$-module. In order to do 
this, we will need to compare the homomorphism $\Delta_\pi$, which is a $\RR_A$-module map, with an $\RR_F$-homomorphism 
denoted $\Delta'_\pi$ which we now describe:
 
Let $M$ be a $\RR_k$-module. We equip $M$ with two $\RR_F$-module structure
\[
\lan a\ran m:=\lan \overline{u}_a\ran m, \ \ \ \ \ \ 
\lan a\ran m:=(-1)^{v(a)}\lan \overline{u}_a\ran m,
\]
where $u_a\in \aa$ is the unique element satisfying $a=u_a\pi^{v(a)}$. 
The $\RR_k$-module $M$ equipped with these $\RR_F$-structures are denoted by
\[
M(v) \ \ \ \ \text{and} \ \ \ \ M\{v\},
\]
respectively. 

\begin{exa}
Consider $\PP(k)$ with the trivial $\GG_k$-action. Then $\PP(k)\{v\}$ is
a (nontrivial) $\RR_F$-module with the definition given by
\[
\lan a\ran[x]:=(-1)^{v(a)}[x].
\]
But $\PP(k)(v)$ becomes  trivial as $\RR_F$-module.
\end{exa}

It is easy to see that the natural maps
\[
\rho_0':\Ind_k^FM \arr M(v), \ \ \ \ 
\lan a\ran \otimes m\mapsto \lan \overline{u}_a\ran m
\]
and 
\[
\rho_\pi':\Ind_k^FM \arr M\{v\}\ \ \ \ \lan a\ran \otimes m
\mapsto (-1)^{v(a)}\lan \overline{u}_a\ran m
\]
are $\RR_F$-homomorphisms (both depending on the choice of the uniformizer
$\pi$). 

\begin{lem}$($\cite[Lemma 6.5]{hutchinson-2017}$)$
Let $M$ be an $\RR_k\half$-module. Then there are natural decompositions 
of $\RR_F\half$-modules
\[
\Ind_k^F M \simeq M(v)\oplus M\{v\} \ \ \ \text{and} \ \ \ \II_F 
\Ind_k^F M \simeq (\II_k M)(v)\oplus M\{v\}.
\]
In other words, $\II_F \Ind_k^F M \simeq \II_k M\oplus M$ made into 
an $\RR_F$-module
by letting $\lan \pi\ran$ acts as the identity on the first factor and as 
multiplication by $-1$ on the second factor.
\end{lem}

Now let 
\[
\delta_0':\RP(F) \arr \widetilde{\RP}(k)(v), \ \ \ \text{and} \ \ \ 
\delta_\pi':\RP(F) \arr \widetilde{\RP}(k)\{v\}
\]
be the composite maps
\[
\RP(F) \overset{S_v}{\larr} \Ind_k^F \widetilde{\RP}(k) 
\overset{\rho_0'}{\larr} \widetilde{\RP}(k)(v), \ \  \text{and} \ \ 
\RP(F) \overset{S_v}{\larr} \Ind_k^F \widetilde{\RP}(k) 
\overset{\rho_\pi'}{\larr} \widetilde{\RP}(k)\{v\},
\]
respectively. These maps define well-defined $\RR_F$-homomorphisms
\[
\delta_0':\RP_1(F) \arr \widetilde{\RP}_1(k)(v), \ \ \ \text{and} \ \ \ 
\delta_\pi':\RP_1(F) \arr \widetilde{\RP}_1(k)\{v\}.
\]
The restriction of these maps to $\II_F\RP_1(F)$ induce
homomorphisms of $\RR_F$-modules
\[
\delta_0':\II_F\RP_1(F) \arr \II_k\widetilde{\RP}_1(k)(v),\ \ \ \ \ \  
\delta_\pi':\II_F\RP_1(F) \arr \widetilde{\RP}_1(k)\{v\}.
\]

\begin{prp}\label{2times}
Let $\eta_\pi$ and $\eta_\pi'$ be the composite maps
\[
\II_F\RP(F) \overset{\delta_\pi}{\larr} 
\widetilde{\RP}(k) \arr \widetilde{\PP}(k),\ \ 
\II_F\RP(F) \overset{\delta_\pi'}{\larr} 
\widetilde{\RP}(k)\{v\} \arr \widetilde{\PP}(k)\{v\},
\]
respectively. Then $\eta_\pi'=-2 \eta_\pi$.
\end{prp}
\begin{proof}
Let $\tilde{\theta}$ be the natural map $\widetilde{\RP}(k) \arr \widetilde{\PP}(k)$.
The maps $\eta_\pi$ and $\eta_\pi'$ are the composite maps
\[
\II_F\RP(F) \overset{S_v}{\larr} \II_F\Ind_k^F\widetilde{\RP}(k) 
\overset{\rho_\pi}{\larr} \widetilde{\RP}(k) 
\overset{\tilde{\theta}}{\larr} \widetilde{\PP}(k), 
\]
and 
\[
\II_F\RP(F) \overset{S_v}{\larr} \II_F\Ind_k^F\widetilde{\RP}(k) 
\overset{\rho_\pi'}{\larr} \widetilde{\RP}(k)\{v\} 
\overset{\tilde{\theta}}{\larr} \widetilde{\PP}(k)\{v\},
\]
respectively.  Let $\Lan a\Ran [x] \in \II_F\RP(F)$ and let  $a=u_a\pi^{v(a)}$, $u_a\in \aa$. 
Then $S_v(\Lan a\Ran [x])=\Lan a\Ran \otimes [x]$. 
If $v(a)$ is even, then 
\[
\rho_\pi(\Lan a\Ran \otimes [x])=0
\]
and
\[
\rho_\pi'(\Lan a\Ran \otimes [x])=\rho_\pi'(\lan a\ran \otimes [x])-
\rho_\pi'(1 \otimes [x])=(-1)^{v(a)}\lan \overline{u}_a\ran [x]- [x]=
\lan \overline{u}_a\ran [x]-[x].
\]
Since the action of $\GG_k$ on $\PP(k)$ is trivial, we have
$\tilde{\theta}(\lan \overline{u}_a\ran [x]-[x])=0$.
If $v(a)$ is odd, then 
\[
\rho_\pi(\Lan a\Ran \otimes [x])=\rho_\pi(\lan a\ran \otimes [x])-
\rho_\pi(1 \otimes [x])=\lan\overline{u}_a\ran [x]-0=\lan\overline{u}_a\ran [x],
\] 
and 
\[
\rho_\pi'(\Lan a\Ran \otimes [x])=\rho_\pi'(\lan a\ran \otimes [x])-
\rho_\pi'(1 \otimes [x])=(-1)^{v(a)}\lan \overline{u}_a\ran [x]-[x]=
-\lan \overline{u}_a\ran [x]-[x].
\] 
Clearly $\tilde{\theta}(\lan\overline{u}_a\ran [x])=[x]$ and 
$\tilde{\theta}(-\lan \overline{u}_a\ran [x]-[x])=-2[x]$. This proves the claim.
\end{proof}

\begin{cor}\label{zeromap1}
The composite
\[
\II_A\RP_1(A) \arr \II_F\RP_1(F) \overset{\delta_\pi'}{\larr} 
\widetilde{\RP}_1(k)\{v\} \arr \widetilde{\PP}(k)\{v\}
\]
is the zero map.
\end{cor}
\begin{proof}
This follows immediately from Proposition \ref{2times} and the fact that the 
composite 
\[
\II_A\RP_1(A) \arr \II_F\RP_1(F) \overset{\delta_\pi}{\larr} 
\widetilde{\RP}_1(k)\]
is the zero map.
\end{proof}

\begin{cor}
Let $k$ be a sufficiently large finite field. Then the sequence
\[
\begin{array}{c}
\II_A\RP_1(A)\half \arr \II_F\RP_1(F)\half \overset{\delta_\pi'}{\larr} {\PP}(k)\{v\}\half \arr 0
\end{array}
\]
is exact.
\end{cor}
\begin{proof}
This follows from Proposition \ref{2times}, Theorem \ref{thm:4.2} and Proposition \ref{hutchinson3}.
\end{proof}

We can now define $\Delta'_\pi$: The $\RR_F$-homomorphism
\[
\Delta_\pi':H_3(\SL_2(F),\z) \arr \widetilde{\RP}_1(k)\{v\}
\]
is the composite $H_3(\SL_2(F),\z) \arr \RP_1(F)\overset{\delta_\pi'}{\larr} \widetilde{\RP}_1(k)\{v\}$.

\begin{prp}\label{diagram1}
If $k$ is sufficiently large, then the diagram of $\RR_A$-modules
\[
\begin{tikzcd}
H_3(\SL_2(A),\z\half) \ar[r, "\lan \pi\ran j_\ast"]\ar[d, "j_\ast"]\ar[dr, "p_\ast"]&
\displaystyle\frac{H_3(\SL_2(F),\z\half)}{j_\ast (H_3(\SL_2(A),\z\half))} \ar[d, "\overline{\Delta}_\pi "]\\
H_3(\SL_2(F),\z\half) \ar[r, "\Delta_\pi' "]&  \RP_1(k)\half
\end{tikzcd}
\]
commutes, where $j:\SL_2(A) \arr \SL_2(F)$ is the natural inclusion and $p:\SL_2(A)\arr \SL_2(k)$
is induced by the quotient map $A\arr k$.
\end{prp}
\begin{proof}
For any local domain $R$
with sufficiently large residue field, we have the exact sequence
\[
\begin{array}{c}
H_3(T_R,\z\half)\arr H_3(\SL_2(R),\z\half) \arr \RB(R)\half \arr 0,
\end{array}
\]
where the left side map is induced by the inclusion $T_R:=\bigg\{{\mtxx a 0 0 {a^{-1}}}|a\in \rr\bigg\} \harr \SL_2(R)$
(see the proof of the refined Bloch-Wigner exact sequence in \cite[Theorem 3.22]{hutchinson2017}).
Thus it is enough to prove the commutativity of the diagram
\[
\begin{tikzcd}
\RB(A)\half \ar[r, "\lan \pi\ran j_\ast"]\ar[d, "j_\ast"]\ar[dr, "p_\ast"]&
\displaystyle\frac{\RB(F)\half }{j_\ast (\RB(A)\half) } \ar[d, "\overline{\delta}_\pi "]\\
\RB(F)\half \ar[r, "\Delta_\pi' "]&  \RP_1(k)\half,
\end{tikzcd}
\]
where $j_\ast:\RB(A) \arr \RB(F)$ is induced by the natural
inclusion map $A \harr F$ and $p_\ast: \RB(A)\arr \RP_1(k)$ is induced by the quotient map $A\arr k$.

If $x=\sum\lan u\ran [a]\in \RB(A)\half $, then $\Delta_\pi'\circ j_\ast (x)=\sum (-1)^{v(u)}\lan \bar{u}\ran[\bar{a}]=
\sum \lan \bar{u}\ran[\bar{a}]=\overline{x}=p_\ast(x)$.  Thus lower triangle of the diagram commutes.
If $\tilde{\beta}:=\lan\pi\ran j_\ast:\RB(A)\arr \RB(F)$, then $\delta_\pi\circ\tilde{\beta}(x)=
\rho_\pi\circ S_v(\lan \pi \ran j_\ast(x))=\rho_\pi\circ S_v(\sum \lan \pi u \ran[a])
= \rho_\pi(\sum \lan \pi u \ran \otimes [\overline{a}])= \sum\lan \overline{u} \ran [\overline{a}]=p_\ast(x)$.
This proves the commutativity of the upper triangle of the diagram.
\end{proof}

We are now in a position to describe the $\RR_F$-module $\coker(\beta)\cong\im(\delta)$:

Let $\overline{\PP}(k)$ denote the cokernel of the composite homomorphism $K_3^\ind(A)\to K_3^\ind(k)\to \PP(k)$. We 
make $\overline{\PP}(k)$ into an $\RR_F$-module by taking the $\chi_v$-twist of the trivial module structure. 

\begin{thm}\label{thm:delta}
There is a natural isomorphism of $\RR_F$-modules 
\[
\im(\delta)\cong \overline{\PP}(k).
\]
\end{thm}

\begin{proof} Since $\overline{\Delta}_\pi$ is an isomorphism by Theorem \ref{mainthm}, it follows from the definition of $\beta$ 
and Proposition \ref{diagram1} that 
\[
\begin{array}{c}
\im(\delta)\cong \coker(\beta)\cong \coker\left(p_*:H_3(\SL_2(A),\z\half)\to \RP_1(k)\half\right)
\end{array}
\]
as an $\RR_F$-module. Now note that the homomorphism
\[
\begin{array}{c}
\II_AH_3(\SL_2(A),\z\half)\cong \II_A\RP_1(A)\half\to \II_k\RP_1(k)\half
\end{array}
\]
is surjective, since the generators $\Lan \bar{u}\Ran g(x)$ of $\II_k\RP_1(k)\half$ can be lifted by Proposition \ref{g(a)}.  The statement of the theorem 
thus follows from the commutative diagram with exact rows:
\[
\begin{tikzcd}
0\ar[r]&\II_AH_3(\SL_2(A),\z\half)\ar[r]\ar[d, two heads]&H_3(\SL_2(A),\z\half)\ar[r]\ar[d, "p_*"]&K_3^\ind(A)\half\ar[r]\ar[d]&0\\
0\ar[r]&\II_k\RP_1(k)\half\ar[r]&\RP_1(k)\half\ar[r]&\PP(k)\half\ar[r]&0.\\
\end{tikzcd}
\]
\end{proof}

We will require the following corollary in our calculation of $\delta$ below:

\begin{cor}\label{cor:delta}
If we endow $H_2(\Gamma_0(\mmm_A),\z)$ with the natural action of $\GG_F$ then 
\[
\begin{array}{c}
\im(\delta)\subset H_2(\Gamma_0(\mmm_A),\z\half)^{\GG_F}.
\end{array}
\]
In particular, the map $C_\pi$ induced from conjugating by $g_\pi\in \tilde{\Gamma}$  is the identity map on $\im(\delta)$.
\end{cor}

\begin{proof}
On the one hand, by Theorem \ref{thm:delta}, $\im(\delta)$ is isomorphic as an $\RR_F$-module to $\overline{\PP}(k)$
with the $\chi_v$-twist of the trivial $\RR_F$-structure. On the other hand,  by Theorem \ref{twisted}, $\im(\delta)$ is a submodule of 
$H_2(\Gamma_0(\mmm_A),\z\half)$ with the Mayer-Vietoris $\RR_F$-structure, which is the $\chi_v$-twist of the 
natural structure. It follows that $\im(\delta)$ is trivial in the natural $\RR_F$-structure. 
\end{proof}

\subsection{The spectral sequence \texorpdfstring{$E_{\bullet,\bullet}(A,R)$}{Lg}}\label{spect-seq}

Next we turn to the explicit calculation of $\delta$, which by the results above, is essentially a homomorphism 
$\overline{\PP}(k)\to H_2(\Gamma_0(\mmm_A),\z)$. In fact we will show that it can be identified with a 
$d^3$-differential in a certain spectral sequence. 

We begin by describing this spectral sequence. 

First, let us describe the general context. Let $G$ be a group and let $L_\bullet$ be a complex of left $G$-modules: 
\[
L_\bullet: \ \ \cdots \larr L_2 \overset{\partial_2}{\larr} L_1 \overset{\partial_1}{\larr} L_0\larr 0.
\]
The $n$-th hyperhomology group  of $G$ with coefficients in $L_\bullet$, denoted by 
$H_n(G, L_\bullet)$,
is defined as the $n$-th homology of the total complex of the double complex 
$F_\bullet \otimes_{G} L_\bullet$, where $F_\bullet\arr \z$ is a  right 
projective resolution of $\z$ over the group ring $\z[G]$. This double complex induces two 
spectral sequences both converging to the hyperhomology groups $H_\bullet(G, L_\bullet)$, as follow:
\[
\mathsf{E}_{p,q}^2=H_p(G, H_q(L_\bullet))\Rightarrow H_{p+q}(G, L_\bullet)
\]
and 
\[
E_{p, q}^1=H_q(G, L_p) \Rightarrow H_{p+q}(G, L_\bullet)
\]
(see \cite[\S 5, Chap. VII]{brown1994}).
By easy analysis of the spectral sequence $\mathsf{E}_{p,q}^2(G)$ we get:

\begin{lem}\label{homology-complex}
Let $L_\bullet$ be exact for $1\leq  i \leq n$. If $M:=H_0(L_\bullet)$, then 
$H_i(G, L_\bullet)\simeq H_i(G, M)$ for $0 \leq  i \leq n$. 
\end{lem}

Let $R$ be a ring. A (column) vector 
${\bf u}={\mtx {u_1} {u_2}}\in R^2$ is said to be unimodular if $u_1R+u_2R=R$.  
Equivalently, ${\bf u}={\mtx {u_1} {u_2}}$ is said to be unimodular 
if there exists a vector ${\bf v}={\mtx {v_1} {v_2}}$ such that the matrix 
$({\bf u}, {\bf v}):={\mtxx {u_1} {v_1} {u_2} {v_2}}$
is  an invertible matrix.   Note that the matrix $({\bf u}, {\bf v})$ is invertible 
if and only if ${\bf u}, {\bf v}$ are a basis of $R^2$.

For any non-negative integer $n$, let $L_n(R^2)$ be the free 
abelian group generated by the set of all $(n+1)$-tuples 
$(R{\bf v_0}, \dots, R{\bf v_n})$, where every ${\bf v_i} \in R^2$ is 
unimodular and any two distinct vectors ${\bf v_i}, {\bf v_j}$ are a basis of $R^2$. 

We consider $L_n(R^2)$ as a left $\GL_2(R)$-module (and so $\SL_2(R)$-module) in a natural way.  
If necessary, we convert these actions to a right actions by the definition $m.g:=g^{-1}m$.

Let us define the $n$-th differential operator $\partial_n : L_n(R^2) \arr L_{n-1}(R^2)$, $n\ge 1$, as an alternating sum 
of face operators which throws away the $i$-th component of generators. Hence we have the complex
\begin{displaymath}
L_\bullet(R^2): \ 
\cdots \larr  L_2(R^2) \overset{\partial_2}{\larr}
L_1(R^2) \overset{\partial_1}{\larr} L_0(R^2) \larr 0.
\end{displaymath}
Let $\partial_{0}=\epsilon: L_0(R^2) \arr \z$ be defined by $\sum_i n_i(R{\bf v_{i}}) \mt \sum_i n_i$. 

\begin{lem}$($\cite[Lemma 3.21]{hutchinson2017}$)$\label{hutchinson1}
If $R$ is a local ring with residue field $k$, then 
the complex $L_\bullet(R^2)$ is exact for $1\leq i < |k|$ and $H_0(L_\bullet(R^2))\simeq \z$.
In particular for $0 \leq i < |k|$,
\[
H_i(\SL_2(R), L_\bullet(R^2))\simeq H_i(\SL_2(R),\z).
\]
\end{lem}

Let $A\arr R$ be a homomorphism of rings. Then  $L_\bullet(R^2) \arr \z$ is a
complex of (left) $\GL_2(A)$-modules (and so $\SL_2(A)$-modules) in a natural way, where $\z$
is considered as trivial module over $\GL_2(A)$. Thus we have the spectral sequences
\[
E_{p, q}^1(A,R)=H_q(\SL_2(A), L_p(R^2)) \Rightarrow H_{p+q}(\SL_2(A), L_\bullet(R^2)),
\]
\[
\ \EE_{p, q}^1(A,R)=H_q(\GL_2(A), L_p(R^2)) \Rightarrow H_{p+q}(\GL_2(A), L_\bullet(R^2)).
\]

When $A=R$ and $A\arr R$ is the identity map, these spectral sequences 
have been studied extensively in \cite{hutchinson-2013}, \cite{hutchinson2017} 
and \cite{suslin1991},  \cite{mirzaii2011},  \cite{mirzaii2017}. 

We suppose now  that the maps  $\aa \arr \rr$ and $\WW_A \arr \WW_R$ are surjective. We set 
\[
I:=\ker(A\arr R).
\]
(Of course, the case of interest in this article is  is the quotient map
$A \arr A/\mmm_A$, where $A$ is a discrete valuation ring with maximal ideal $\mmm_A$.)

To study the spectral sequence $E_{\bullet,\bullet}(A,R)$, we must study the action 
of $\SL_2(A)$ on the sets of basis of $L_i(R^2)$ for $0\leq i\leq 4$. 
Let
\[
{\bf e_1}:={\mtx 1 0}, \ \ {\bf e_2}:={\mtx 0 1} \in R^2.
\]
It is easy to see that $\SL_2(A)$ acts transitively on 
the sets of generators of $L_i(R^2)$ for $i=0,1$. Here we use the subjectivity 
of the map $\aa \arr \rr$. We choose $({\bf e_1}R)$ and  
$({\bf e_1}R,{\bf e_2}R)$ as representatives of the orbit of the generators
of $L_0(R^2)$ and $L_1(R^2)$, respectively.
The orbits of the action of  $\SL_2(A)$ on $L_2(R^2)$, $L_3(R^2)$ and $L_4(R^2)$ are represented by
\[
\lan a\ran[\ ]:=({\bf e_1}R, {\bf e_2}R, ({\bf e_1}+a {\bf e_2})R), \  \  \  \  \  \  
\lan a\ran\in \GG_R,
\] 
\[
\lan a\ran[x]:=({\bf e_1}R,{\bf e_2}R, ({\bf e_1}+a {\bf e_2})R, 
({\bf e_1}+ax{\bf e_2})R,  \  \  \  \   
\lan a\ran\in \GG_R, x\in \WW_R,
\]
and 
\[
\lan a\ran[x,y]:=({\bf e_1}R,{\bf e_2}R, ({\bf e_1}+a{\bf e_2})R, 
({\bf e_1}+ax{\bf e_2})R, ({\bf e_1}+ay{\bf e_2})R),  \ \  \lan a\ran\in \GG_R, x,y,x/y\in \WW_R,
\]
respectively. Therefore 
\[
L_0(R^2)\simeq \Ind _{\Gamma_0(I)}^{\SL_2(A)}\z, \ \ \ \ \  L_1(R^2)\simeq 
\Ind _{\Gamma_1(I)}^{\SL_2(A)}\z, \ \ \ \ \ 
L_2(R^2)\simeq \bigoplus_{\lan a \ran \in \GG_R} \Ind _{\Gamma_2(I)}^{\SL_2(A)}\z \lan a \ran[\ ],
\]
\[
L_3(R^2)\simeq \bigoplus_{\lan a\ran \in \GG_R}
\bigoplus_{x\in \WW_R}\Ind_{\Gamma_2(I)}^{\SL_2(A)}\z\lan a\ran [x], \ \  \ \ \
L_4(R^2)\simeq \bigoplus_{\lan a\ran \in \GG_R}
\bigoplus_{x,y,x/y\in \WW_R}\Ind_{\Gamma_2(I)}^{\SL_2(A)}\z\lan a\ran[x,y],
\]
where
\begin{align*}
&\Gamma_0(I):=\bigg\{ {\mtxx a b c d} \in \SL_2(A): c \in I\bigg\}, \\
&\Gamma_1(I):=\bigg\{ {\mtxx a b c d} \in \SL_2(A):  b, c \in I\bigg\},\\
&\Gamma_2(I):=\bigg\{ {\mtxx a b c d} \in \SL_2(A): b, c,  a-d \in I\bigg\}.
\end{align*}
Thus  by Shapiro's lemma we have
\[
E_{0,q}^1(A,R) \simeq H_q(\Gamma_0(I),\z), \
E_{1,q}^1(A,R) \simeq H_q(\Gamma_1(I),\z),\ 
E_{2,q}^1(A,R) \simeq \bigoplus_{\lan a\ran\in \GG_R} H_q(\Gamma_2(I),\z),
\]
\[
E_{	3,q}^1(A,R)\simeq \bigoplus_{\lan a\ran \in \GG_R}\bigoplus_{x\in \WW_R}H_q(\Gamma_2(I),\z), \ \ \ \
E_{4,q}^1(A,R)\simeq \bigoplus_{\lan a\ran \in \GG_R}
\bigoplus_{x,y,x/y\in \WW_R}H_q(\Gamma_2(I),\z).
\]
In particular $E_{0,0}^1(A,R)\simeq \z$, $E_{1,0}^1(A,R)\simeq \z$, 
$E_{2,0}^1(A,R)\simeq \z[\GG_R]$.

If $Z_1$ is the free abelian group generated by the symbols $[x]$, $x \in \WW_R$, and $Z_2$ is the 
free abelian group generated by the symbols $[x, y]$, $x,y, x/y \in \WW_R$, then in case of $q=0$, we have
\[
E_{3,0}^1(A,R)\simeq \z[\GG_R]Z_1, \ \ \ \ \ \ \  \ E_{4,0}^1(A,R)\simeq \z[\GG_R]Z_2.
\]

Now we study the differentials of the spectral sequence:

It is not difficult to see that 
\[
d_{1, q}^1=H_q(\sigma) - H_q(\inc),
\]
where $\sigma: \Gamma_1(I) \arr \Gamma_0(I)$ is conjugation by $w={\mtxx 0 {-1} 1 0}$,
i.e. $\sigma(X)= wX w^{-1}$. Furthermore
\[
d_{2,q}^1|_{\lan a\ran \otimes H_q(\Gamma_2(I),\z)}
=H_q(\eta_a)-H_q(\eta_a')+H_q(\inc),
\]
where $\eta_a,\eta_a': \Gamma_2(I) \arr \Gamma_1(I)$ are conjugation by
${\mtxx {-a} 1 {-1} 0 }$ and ${\mtxx 1 {-a^{-1}} 0 1}$, respectively. 
In particular $d_{2,0}^1:\z[\GG_R] \arr \z$ is the usual augmentation 
map $\sum n_i\lan a_i\ran \mapsto \sum n_i$.

The action of $\GL_2(A)$ on $L_\bullet(R^2)$ and the extension 
\[
1 \arr \SL_2(A)  \arr \GL_2(A)  \overset{\det}{\arr} A^{\times} \arr 1,
\]
induces an action 
of $\aa$ on $E_{\bullet,0}^1(A,R)=L_\bullet(R^2)_{\SL_2(A)}$.
Since $\aa I_2=Z(\GL_2(A))$ acts trivially on $E_{p,0}^1(A,R)$,
$(\aa)^2=\{a^2|a\in \aa\}$ also will act trivially on $E_{\bullet,0}^1(A,R)$. 
Thus $E_{p,0}^1(A,R)$, $p\geq 0$, has a 
natural $\GG_A$-module structure. Moreover the differentials
$d_{p,0}^1: E_{p,0}^1(A,R) \arr E_{p-1,0}^1(A,R)$ are $\GG_A$-homomorphisms.

By a direct calculation one sees that the $\GG_A$-homomorphism
$d_{3,0}^1:\z[\GG_R]Z_1 \arr \z[\GG_R]$ is given by 
\[
[x] \mapsto \Lan x\Ran \Lan 1-x\Ran,
\]
while the $\GG_A$-homomorphism $d_{4,0}^1:\z[\GG_R]Z_2 \arr \z[\GG_R]Z_1$ 
is given by
\[
[x,y] \mapsto Y_{x,y}=[x]-[y]+\lan x\ran\bigg[\frac{y}{x}\bigg]-
\lan x^{-1}-1\ran\Bigg[\frac{1-x^{-1}}{1-y^{-1}}\Bigg] - 
\lan 1-x\ran\Bigg[\frac{1-x}{1-y}\Bigg].
\]

Finally from the above calculations we have 
\[
E_{1,0}^2(A,R)=0, \ \ \ \  E_{2,0}^2(A,R)=I(R), \ \ \ \ E_{3,0}^2(A,R)\simeq \RP_1(R),
\]
where 
\[
I(R):=\II_R/\lan\ \Lan a\Ran \Lan 1-a\Ran:  a \in \WW_R\ran
\]
is called the {\it fundamental ideal} of $R$. (In fact, when $R$ is a local ring, it is the fundamental ideal in the Grothendieck-Witt ring
\[
\mathrm{GW}(R)\cong \z[\GG_R]/\lan \Lan u \Ran\Lan 1-u\Ran\ | u\in \WW_R\ran .)
\]

The commutative diagram of ring homomorphism
\[
\begin{CD}
A &@>{\id_A}>>&A\\
@V{\id_A}VV & &@VVV\\
A &@>>>&R\\
@VVV & &@VV{\id_R}V\\
R &@>{\id_R}>>&R\\
\end{CD}
\]
induces morphisms of spectral sequences 
\[
E_{\bullet,\bullet}(A,A) \arr E_{\bullet,\bullet}(A,R) 
\arr E_{\bullet,\bullet}(R,R)
\]
which give us the commutative diagram
\[
\begin{CD}
I(A) &@>{d_{2,0}^2(A,A)}>>& E_{0,1}^2(A,A)\\
@VVV & &@VVV\\
I(R) &@>{d_{2,0}^2(A,R)}>>&E_{0,1}^2(A,R)\\
@V\id_{I(R)}VV & &@VVV\\
I(R) &@>{d_{2,0}^2(R,R)}>>&E_{0,1}^2(R,R).
\end{CD}
\]
Note that 
\[
E_{0,q}^1(A,A)=H_q(B_A,\z), \ \  E_{1,q}^1(A,A)=H_q(T_A,\z), \ \
E_{2,q}^1(A,A)=\bigoplus_{\lan a\ran \in \GG_A}H_q(\mu_2(A)I_2,\z),
\]
where 
\[
T_A=\bigg\{{\mtxx a 0 0 {a^{-1}}}|a\in \aa\bigg\}, \ \ \ \ B_A:=\bigg\{{\mtxx a b 0 {a^{-1}}}|a\in \aa, b\in A\bigg\}.
\]
It is easy to see that $E_{0,1}^2(A,A)\simeq \GG_A \oplus H_0(\aa,A)$ and 
$E_{0,1}^2(R,R)\simeq \GG_R \oplus H_0(\rr, R)$, where the action of $\aa$ on $A$
(reps. $\rr$ on $R$) is given by $(a,b)\mt a^2b$.

By direct calculation one can show that 
both of the maps $d_{2,0}^2:I(A) \arr \GG_A\oplus H_0(\aa,A)$
and $d_{2,0}^2:I(R) \arr \GG_R\oplus H_0(\rr,R)$ are given by 
$\Lan a \Ran \mapsto \lan a \ran$ (see \cite[Lemma 5]{mazzoleni2005}).
Thus $E_{2,0}^3(A,A)\simeq I^2(A)$ and $E_{2,0}^3(R,R)\simeq I^2(R)$. 
Now from the surjectivity of $I(A) \arr I(R)$ and the above diagram
we obtain the isomorphism
\[
E_{2,0}^3(A,R)\simeq I^2(R).
\]

\subsection{The second homology of  \texorpdfstring{$\SL_2$}{Lg} of a local ring}

\begin{prp}\label{exact1}
Let $A$ be a local ring whose residue field has
at least four elements. Let $I$ be a proper ideal of $A$ and set $R:=A/I$.
Then we have a natural exact sequence
\[
H_2(\Gamma_0(I),\z) \arr H_2(\SL_2(A),\z)\arr I^2(R) \arr 0.
\]
\end{prp}
\begin{proof}
By Lemma \ref{hutchinson1}, $H_i(\SL_2(A), L_\bullet)\simeq 
H_i(\SL_2(A),\z)$ for $0\leq i \leq 3$.
Moreover the natural map $\aa \arr \rr$ is surjective. 

Thus to prove the claim it is enough to prove that $E_{1,1}^2(A,R)=0$. This is the
homology of the complex
\[
\begin{CD}
\bigoplus_{\lan a\ran\in \GG_R}H_1(\Gamma_2(I),\z) @>{d_{2,1}^1(A,R)}>>
H_1(\Gamma_1(I),\z) @>{d_{1,1}^1(A,R)}>> H_1(\Gamma_0(I),\z).
\end{CD}
\]
Let $\Gamma(I):=\bigg\{{\mtxx a b c d}\in \SL_2(A):a-1,d-1,b,c\in I\bigg\}$.
From the commutative diagram of extensions
\[
\begin{CD}
1 \larr &\ \Gamma(I)&\ \larr &\ \Gamma_2(I)&\ \larr &\ \mu_2(R)I_2&\ \larr 1 \\
       &@VVV        @VVV         @VVV     &        \\
1 \larr &\ \Gamma(I)&\ \larr &\ \Gamma_1(I)&\ \larr &\ T_R& \ \larr 1 \\
       &@VVV        @VVV         @VVV     &        \\
1 \larr &\ \Gamma(I)&\ \larr &\ \Gamma_0(I)&\ \larr &\ B_R&\ \larr 1 \\
\end{CD}
\]
and the morphism of spectral sequences $E_{\bullet, \bullet}(A,R) 
\arr E_{\bullet, \bullet}(R,R)$
we obtain the commutative diagram
\[
\begin{tikzcd}
\bigoplus_{\lan a\ran\in \GG_R}H_1(\Gamma(I),\z)\ar[r]\ar[d] &
H_1(\Gamma(I),\z)_{T_R} \ar[r] \ar[d] & H_1(\Gamma(I),\z)_{B_R}\ar[d]\\
\bigoplus_{\lan a\ran\in \GG_R}H_1(\Gamma_2(I),\z)\ar[r, "{d_{2,1}^1(A,R)}"]\ar[d] &
H_1(\Gamma_1(I),\z) \ar[r, "{d_{1,1}^1(A,R)}"] \ar[d]& H_1(\Gamma_0(I),\z) \ar[d]\\
\bigoplus_{\lan a\ran\in \GG_R}H_1(\mu_2(R),\z)\ar[r, "{d_{2,1}^1(R,R)}"] &
H_1(T_R,\z)\ar[r, "{d_{1,1}^1(R,R)}"] & H_1(B_R,\z),
\end{tikzcd}
\]
where the maps on top row is induced by the maps on the middle row.
Now the triviality of $E_{1,1}^2(A,R)$ follows from the following facts:
\par (i) The natural map $\Gamma_0(I) \arr T_R$ induces the isomorphism
\[
H_1(\Gamma_0(I),\z) \simeq H_1(T_R,\z)\simeq \rr.
\]
\par (ii) The elements of 
$H_1(\Gamma(I),\z)_{T_R}=\Gamma(I)/[\Gamma(I),\Gamma_1(I)]$ are represented
 by matrices  ${\mtxx a 0 0 {a^{-1}}}$, $a-1\in I$.

In fact (i) implies that 
$H_1(\Gamma_0(I),\z)=\Gamma_0(I)/[\Gamma_0(I),\Gamma_0(I)]$ injects into  $H_1(B_R,\z)$.
Now by applying the Snake lemma to the diagram
\[
\begin{tikzcd}
&H_1(\Gamma(I),\z)_{T_R} \ar[r]\ar[d] & H_1(\Gamma_1(I),\z) \ar[r] \ar[d, "{d_{1,1}^1(A,R)}"] & 
H_1(T_R,\z) \ar[d, "{d_{1,1}^1(R,R)}"] \ar[r]&  0\\
 0\ar[r]  &0 \ar[r]                  &  H_1(\Gamma_0(I),\z) \ar[r] &  H_1(B_R,\z) & 
\end{tikzcd}
\]
we obtain the exact sequence
\[
H_1(\Gamma,\z)_{T_R} \arr \ker(d_{1,1}^1(A,R)) \arr H_1(\mu_2(R),\z) \arr 0.
\]
The map $d_{2,1}(A,R)|_{H_1(\Gamma(I),\z)} H_1(\Gamma(I),\z) 
\arr H_1(\Gamma(I),\z)_{T_R}$
takes the element represented by ${\mtxx a 0 0 {a^{-1}}}$ to the 
element represented 
by the same matrix, thus by (ii) this map is surjective. Therefore the claim follows from 
the commutativity of the above diagram.
Now we prove (i) and (ii).
\par {\sf Proof of (i)}: Let $\Lambda_0(I)$ be the kernel of the natural map
\[
\Gamma_0(I) \arr T_R, \ \ \ \ {\mtxx a b c d}\mapsto 
{\mtxx {\bar{a}} 0 0 {\bar{a}^{-1}}}.
\]
Since $T_R$ is abelian, 
$[\Gamma_0(I),\Gamma_0(I)]\subseteq \Lambda_0(I)$. 
Let ${\mtxx a b c d}\in \Lambda_0(I)$. Since $A$ is local, there is 
$z \in (a-1)A \subset I$ such that $x:=a+bz\in \aa$. Now if $y:=c+dz\in I$,
then
\[
{\mtxx a b c d}={\mtxx x 0 0 {x^{-1}}}
{\mtxx 1 0 {xy} 1}{\mtxx 1  {x^{-1}b} 0 1}
{\mtxx 1 0 {-z} 1}.
\]
Since $x+I=a+I=1+I$, $x-1\in I$. Let $x=1+t$ for some $t \in I$. Then
\[
{\mtxx x 0 0 {x^{-1}}}= {\mtxx 1 0 {-x^{-1}t} 1} {\mtxx 1 1 0 1}{\mtxx 1 0 t 1}
{\mtxx 1 {-x^{-1}} 0 1}.
\]
Let $\lambda \in \aa$ such that $\lambda^2-1 \in \aa$. This is possible since
the residue field of $A$ has at least four elements. Now the above formulas
together with the commutator formulas
\[
{\mtxx 1 s 0 1}=\bigg[{\mtxx \lambda 0 0 {\lambda^{-1}}},
{\mtxx 1 {s/(\lambda^2-1)} 0 1}\bigg],
\]
\[
{\mtxx 1 0 s 1}=\bigg[{\mtxx \lambda 0 0 {\lambda^{-1}}},
{\mtxx 1 0 {s/(\lambda^2-1)} 1}\bigg],
\]
$s\in A$, imply that 
$\Lambda_0(I) \subseteq [\Gamma_0(I),\Gamma_0(I)]$. Thus 
\[
H_1(\Gamma_0(I),\z)\simeq T_R\simeq H_1(T_R,\z).
\]

\par {\sf Proof of (ii)}: First note that
\[
H_1(\Gamma(I),\z)_{T_R}=H_1(\Gamma(I),\z)_{\Gamma_1(I)}\simeq
 \Gamma(I)/[\Gamma(I), \Gamma_1(I)].
\]
If ${\mtxx a b c d}\in \Gamma(I)$, then as in above we find $z \in I$, such that
\[
{\mtxx a b c d}={\mtxx x 0 0 {x^{-1}}}{\mtxx 1 0 {xy} 1}
{\mtxx 1 {x^{-1}b} 0 1}
{\mtxx 1 0 {-z} 1},
\]
where $x:=a+bz\in \aa$ and $y:=c+dz\in I$. Now the claim (ii) 
follows from the above commutator formulas.
\end{proof}

\subsection{The group \texorpdfstring{$E_{3,0}^3(A, R)$}{Lg} and the differential \texorpdfstring{$d_{3,0}^3(A,R)$}{Lg}}

The morphism of spectral sequences 
$E_{\bullet,\bullet}(A,A) \arr E_{\bullet,\bullet}(A,R)$
induces the commutative diagram
\[
\begin{CD}
\RP_1(A) &@>{d_{3,0}^2(A,A)}>>& E_{1,1}^2(A,A)\\
@VVV & &@VVV\\
\RP_1(R) &@>{d_{3,0}^2(A,R)}>>&E_{1,1}^2(A,R).\\
\end{CD}
\]
Since $\WW_A \arr \WW_R$ is surjective, $ \RP_1(A)  \arr \RP_1(R) $ 
is surjective. It is not difficult to show that $E_{1,1}^2(A,A)=0$. Thus
$d_{3,0}^2(A,A)$ is trivial. This implies that
\[
d_{3,0}^2(A,R):\RP_1(R) \arr E_{1,1}^2(A,R)
\]
is trivial too. Therefore 
\[
E_{3,0}^3(A,R)\simeq \RP_1(R).
\]

Now we would like to calculate the differential $d_{3,0}^3(A,R):\RP_1(R) \arr E_{0,2}^3(A,R)$. (Note that $E_{0,2}^3(A,R)$ 
is a quotient of $H_2(\Gamma_0(I),\z)$.) For this we need to put an extra condition on $A$. 

\begin{prp}$($\cite[Proposition 3.19]{hutchinson2017}$)$\label{homology2}
Let $A$ be a local domain with residue field $k$. If $k$ is finite
we assume that it has $p^d$ elements such that   
$(p-1)d > 4$. Then the natural map 
\[
H_2(\inc): H_2(T_A,\z) \arr H_2(B_A,\z)
\]
is an isomorphism.
\end{prp}

So let the natural map $H_2(\inc): H_2(T_A,\z) \arr H_2(B_A,\z)$ be an isomorphism. By analyzing the morphism of spectral 
sequences
\[
E_{\bullet,\bullet}(A,A) \arr \EE_{\bullet,\bullet}(A,A),
\]
one can show (see the proof of \cite[Proposition 6.1]{hutchinson2019}) that the diagram 
\[
\begin{tikzcd}
\RP_1(A) \ar[r, "{d_{3,0}^2(A,A)}"] \ar[d]& \displaystyle \frac{\aa \wedge \aa}{\aa \wedge \mu_2(A)}\ar[d]\\
\PP(A) \ar[r, "\lambda"] &S_\z^2(A)
\end{tikzcd}
\]
is commutative, where the vertical map on the right side is injective and is given by 
\[
a\wedge b \mapsto 2(a\otimes b).
\]
Under the map $\RP_1(A) \arr \PP(A)$, $g(a)=p_{-1}^+[a]+\Lan 1-x\Ran \psi_1(a)$ maps to $2[a]$.
This shows that in the above  diagram $d_{3,0}^3(A,A)$ maps $g(a)$ to $ (1-a)\wedge a$.
(Note that when $A$ is a local ring where its residue field has more than 10 elements, then the 
set $\{g(a):a\in \WW_A\}$ generates $\RP_1(A)[\frac{1}{2}]$.) Now by considering the commutative diagram 
\[
\begin{tikzcd}
\RP_1(A) \ar[r, "{d_{3,0}^2(A,A)}"] \ar[d]& \displaystyle \frac{\aa \wedge \aa}{\aa \wedge \mu_2(A)}\ar[d]\\
\RP_1(R) \ar[r, "{d_{3,0}^2(A,R)}"] &H_2(\Gamma_0(I),\z)/K
\end{tikzcd}
\]
we see that under the map $d_{3,0}^2(A,R)$, 
\[
g(a) \mapsto {\rm \bf{c}} \bigg({\mtxx {(1-a)} 0 0 {(1-a)^{-1}} }, 
{\mtxx a 0 0 {a^{-1}}}\bigg) \pmod K.
\]
For $a,b \in \aa$, ${\rm \bf{c}} \bigg({\mtxx a 0 0 {a^{-1}} }, {\mtxx b 0 0 {b^{-1}}}\bigg)$ is
the image of $a\wedge b$ under the composite
\[
\begin{array}{c}
\aa \wedge \aa \overset{\simeq}{\larr} T_A \wedge T_A \overset{\simeq}{\larr}H_2(T_A,\z) 
\overset{H_2(\inc)}\larr H_2(\Gamma_0(I),\z).
\end{array}
\]


We deduce: 

\begin{prp}\label{exact2}
Let $A$ be a local domain with residue field $k$. If $k$ is finite we assume that it has $p^d$ elements such that 
$(p-1)d > 4$. Let $I$ be a proper ideal of $A$ and set $R:=A/I$.  Then there is a natural exact sequence 
\[
H_3(\SL_2(A),\z) \larr \RP_1(R) \overset{\alpha}{\larr} 
\displaystyle\frac{H_2(\Gamma_0(I),\z)}{K} \larr H_2(\SL_2(A),\z)\larr I^2(R) \arr 0,
\]
where $K$ is a certain subgroup of $H_2(\Gamma_0(I),\z)$ and
\[
\alpha(g(a))={\rm \bf{c}} \bigg({\mtxx {(1-a)} 0 0 {(1-a)^{-1}} }, 
{\mtxx a 0 0 {a^{-1}}}\bigg)\pmod K.
\]
\end{prp}
\begin{proof}
By Lemma \ref{hutchinson1}, $H_i(\SL_2(A),L_\bullet)\simeq H_i(\SL_2(A),\z)$,  for $0\leq i\leq 3$.
Moreover by Proposition \ref{homology2}, $H_2(T_A,\z)\simeq H_2(B_A,\z)$.                                                                                                                                                                                                                                                                                                                                                                  
Since the maps $\aa \arr \rr$ and $\WW_A\arr \WW_R$ are surjective, the claim follows of an
easy analysis of the spectral sequence $E_{\bullet,\bullet}(A,R)$.
\end{proof}

\subsection{Comparing two exact sequences}

Again, let $A$ be a discrete valuation ring with maximal ideal $\mmm_A$. Let $k$, $F$, $v=v_A$ and $\pi$
be as in Section \ref{sec:RSCG}.

By Theorem \ref{MV1} we have the Mayer-Vietoris exact sequence 
of $\GG_F$-modules
\[
\begin{CD}
H_3(\SL_2(A),\z)\oplus H_3(\SL_2(A),\z)
@>{ j_\ast+\lan\pi\ran j_\ast}>> H_3(\SL_2(F),\z)
\overset{\delta}{\larr} H_2(\Gamma_0(\mmm_A),\z)
\end{CD}
\]
\[
\begin{CD}
@>{(i_\ast,-i_\ast(\lan \pi \ran\bullet))}>> H_2(\SL_2(A),\z)\oplus 
H_2(\SL_2(A),\z) @>{j_\ast+\lan\pi\ran j_\ast}>>
H_2(\SL_2(F),\z)
\end{CD}
\]
From this and Proposition \ref{exact1} we obtain the exact sequence
\[
H_3(\SL_2(A),\z)
\overset{\lan\pi\ran j_\ast}{\larr} 
\frac{H_3(\SL_2(F),\z)}{j_\ast H_3(\SL_2(A),\z)}\overset{\delta}{\larr} 
H_2(\Gamma_0(\mmm_A),\z) \overset{i_\ast}{\larr} H_2(\SL_2(A),\z) \arr I^2(k)\arr 0.
\]
On the other hand, by Proposition \ref{exact2}, we have the exact sequence
\[
\begin{CD}
H_3(\SL_2(A),\z) \arr \RP_1(k) @>{{d_{3,0}^3(A,k)}}>> E_{0,2}^3(A,k) 
\arr H_2(\SL_2(A),\z) \arr I^2(k) \arr 0.
\end{CD}
\]

The following proposition compares these two exact sequences. 

\begin{prp}\label{main-diagram}
If $k$ is sufficiently large then the diagram with exact  rows 
\[
\begin{tikzcd}
H_3(\SL_2(A),\z\half) \ar[r, "{\lan\pi\ran j_\ast}"] \ar[d, "="] &
\displaystyle\frac{H_3(\SL_2(F),\z\half)}{j_\ast H_3(\SL_2(A),\z\half)} 
\ar[r, "\delta "] \ar[d, "\overline{\Delta}_\pi "] &H_2(\Gamma_0(\mmm_A),\z\half) \ar[d, two heads]\\
H_3(\SL_2(A),\z\half) \ar[r] &  \RP_1(k)\half \ar[r, "{d_{3,0}^3(A,k)} "] & E_{0,2}^3(A,k)\half 
\end{tikzcd}
\]
commutes. 
\end{prp}   
\begin{proof}
The commutativity of the left square of the diagram 
is proved in Proposition \ref{diagram1}.
To prove the commutativity of other square we fix some notations.

Let $R$ be a local ring with sufficiently large residue field and let $G$ be a group
that acts on the complex $L_\bullet(R^2)$. Let $P_\bullet \arr \z$ be a 
free resolution of $\z$ over $G$. This is also a free resolution of $\z$ 
over any subgroup $H$ of $G$.

For any $\z\half[G]$-module $M$ and any subgroup $H$ of $G$, Let $T_\bullet^R(H,M)$ 
be the total complex of the double complex 
\[
D_{\bullet,\bullet}^{R, H}
:=P_\bullet\otimes_H\Big(L_\bullet(R^2)\otimes_\z M\Big).
\]

On the one hand the hyperhomology of $H$ with coefficients in the complex 
$L_\bullet(R^2)\otimes_\z M$ is the homology of the total complex 
$T_\bullet^R(H,M)$. Thus  by Lemma \ref{homology-complex},
\[
H_i(T_\bullet^R(H,M))\simeq H_i(H, L_\bullet(R^2)\otimes_\z M) \simeq H_i(H,M),
\ \ \ \ \text{for}  \ \ \ \ 0\leq i\leq 3.
\]
On the other hand from the double complex $D_{\bullet,\bullet}^{R, H}$ 
we obtain the spectral sequence 
\[
\mathbf{E}_{p,q}^1=H_q(H, L_p(R^2) \otimes_\z M)\Rightarrow 
H_{p+q}(H, L_\bullet(R^2) \otimes_\z M).
\]
For some of the basic properties of the total complex $T_\bullet^R(H,M)$ which 
we will use, see \cite[Section 8.2]{hutchinson2015}

Let $G:=\SL_2(F)$, $G_0=\SL_2(A)$, $G_1=\SL_2(A)^{g_\pi}$ and
$\Gamma=\Gamma_0(\mmm_A)$. Consider the exact sequence of $\RR_F$-modules
\[ 
\begin{array}{l}
0 \arr \z\half[G/\Gamma] \arr \z\half[G/G_0] \oplus \z[G/G_1] \arr \z\half \arr 0
\end{array}
\] 
obtained in Subsection \ref{MV0}. From this we obtain the exact sequence of complexes 
\[ 
\begin{array}{l}
\hspace{-0.2 cm}
0\! \arr \!
T_\bullet^F(G, \z\half [G/\Gamma]) \arr T_\bullet^F(G, \z\half[G/G_0]) \oplus T_\bullet^F(G, \z\half[G/G_1]) 
\arr T_\bullet^F(G, \z\half \!\arr \!0.
\end{array}
\]
The long exact sequence associated to this exact sequences of complexes 
gives us the Mayer-Vietoris exact sequence of $\RR_F$-modules
\[
\begin{array}{c}
H_3(G_0,\z\half) \oplus H_3(G_1,\z\half) \overset{\beta}{\arr} H_3(G,\z\half) 
\overset{\delta}{\arr} H_2(\Gamma,\z\half) \\
\\
\overset{\alpha}{\arr} H_2(G_0,\z\half) \oplus H_2(G_1,\z\half) 
\end{array}
\]
(see Theorem \ref{MV1}) which from it we obtain the upper exact sequence in our diagram:
\[
\begin{array}{l}
H_3(\SL_2(A),\z\half) \overset{\tilde{\beta}}{\larr} {\displaystyle
\frac{H_3(\SL_2(F),\z\half)}{H_3(\SL_2(A),\z\half)} } \overset{\delta}{\larr} 
H_2(\Gamma,\z\half) 
\overset{i_\ast}{\larr} H_2(\SL_2(A),\z\half).
\end{array}
\]

Consider the commutative  diagram of complexes
\[  
\begin{array}{c}
\begin{tikzcd}
0\! \arr\! T_\bullet^F(G, \z\half[G/\Gamma])\!\!\ar[r] &\!\! T_\bullet^F(G, \z\half[G/G_0]) &\!\!\!\!\!\!\!\!\!\!\!\!\!\!\!\!\!\!
\!\oplus\! \ T_\bullet^F(G, \z\half[G/G_1])\! \arr\! T_\bullet^F(G, \z\half)\! \arr \!0\\
T_\bullet^F(\Gamma,\z\half)\ar[u]\ar[r] & T_\bullet^F(G_0, \z\half)\ar[u,"{(i_\bullet, 0)}"] &\\
T_\bullet^A(\Gamma,\z\half)\ar[u]\ar[r]\ar[d] & T_\bullet^A(G_0, \z\half)\ar[u] \ar[d] &\\
T_\bullet^k(\Gamma,\z\half) \ar[r] & T_\bullet^k(G_0,\z\half).&
\end{tikzcd}
\end{array}
\]
The double complex which is used to construct $T_\bullet^k(G_0,\z\half)$ is used 
to construct the spectral sequence $E_{\bullet,\bullet}(A,k)\half$. This spectral sequence is used 
to prove the exactness of the lower sequence in our diagram.

Let  $\omega \in \im(\delta)=\ker(i_\ast)$ and let $\omega$ be represented by 
$z\otimes 1 \in P_2\otimes_{\Gamma} \z\half$.  We know that
\[
\begin{array}{c}
\omega\in H_2(\Gamma,\z\half)=H_2(P_\bullet\otimes_{\Gamma} \z\half)\simeq 
H_2(T_\bullet^A(\Gamma,\z\half)).
\end{array}
\]
Let $\tilde{z}=(0,0,z\otimes {\bf e_1}A )$ be its corresponding element in 
\[
\begin{array}{c}
T_2^A(\Gamma,\z\half)\!=\!\bigoplus_{i=0}^2 P_i \otimes_{\Gamma} L_{2-i}(A^2)\half.
\end{array}
\]
Since 
\[
\begin{array}{c}
H_2(\Gamma,\z\half)\simeq H_2(T_\bullet^A(\Gamma,\z\half) )\simeq 
H_2(T_\bullet^F(\Gamma,\z\half) ),
\end{array}
\]
we may also assume that
\[
\begin{array}{c}
\tilde{z}=(0,0,z\otimes {\bf e_1}F ) \in T_2^F(\Gamma,\z\half)=
\bigoplus_{i=0}^2 P_i \otimes_{\Gamma} L_{2-i}(F^2)\z\half.
\end{array}
\]
In the above diagram $\tilde{z}\in T_2^F(\Gamma,\z\half)$ maps to
$\tilde{z}\otimes \Gamma\in T_2^F(G,\z\half[G/\Gamma])$ which represent
the element $\omega \in H_2(\Gamma,\z\half)=H_2(T_\bullet^F(G, \z\half[G/\Gamma])$.

Since under the map 
\[
\begin{array}{c}
H_2(T_\bullet^A(\Gamma,\z\half))=H_2(\Gamma,\z\half) \overset{i_\ast}{\larr}
H_2(G_0, \z\half)= H_2(T_\bullet^A(G_0,\z\half)),
\end{array}
\]
$\omega$ maps to zero, it follows that the image of $\tilde{z}$ in
$T_2^A(G_0,\z\half)$ is a boundary. Thus there exists
$x=(x_0,x_1,x_2,x_3)\in T_3^A(G_0,\z\half)=\bigoplus_{i=0}^3 P_i\otimes _{G_0} L_{3-i}(A^2)\half$ such that
$d_3(x)=\tilde{z}$, where $d_3=(d^h+(-1)^p d^v)$. More precisely,
\[
(0,0,z\otimes {\bf e_1}A)=d_3(x)=(
d^h(x_1)+d^v(x_0), d^h(x_2)-d^v(x_1),d^h(x_3)+d^v(x_2)).
\]
From $d^h(x_1)+d^v(x_0)=0$ it follows that $d_{3,0}^1(x_0)=0$ and thus
$\overline{x}_0\in E_{3,0}^2(A,k)\half$. From $d^h(x_2)-d^v(x_1)=0$, it follows 
that $d_{3,0}^2(\overline{x}_0)=0$ and thus $x_0$ represent an element 
$\alpha$ of $E_{3,0}^3(A,k)\half=\RP_1(k)\half$. Through the natural morphism 
$D_{\bullet,\bullet}^{A,G_0} \arr D_{\bullet,\bullet}^{k,G_0}$, 
$x_0$ represent an element $\alpha$ of $E_{3,0}^3(A,k)\half=\RP_1(k)\half$,
such that $d_{3,0}^3(\alpha)$ is the homology class in $E_{0,2}^2(A,k)\half=
H_2(\Gamma,\z\half)/K$ represented by $z \otimes {\bf e_1}A$. Thus
\[
\begin{array}{c}
d_{3,0}^3(\alpha)=\omega +K \in E_{0,2}^2(A,k)\half.
\end{array}
\]

The conjugation isomorphism $C_\pi: G_0 \arr G_1$ induces the isomorphism of complexes
\[
\begin{array}{c}
{C_\pi}_\bullet: T_\bullet^F(G_0,\z\half) \arr T_\bullet^F(G_1,\z\half).
\end{array}
\]
Under this the element $\tilde{z}=(0,0,z\otimes {\bf e_1}F)\in T_2^F(G_0,\z\half)=\bigoplus_{i=0}^2 P_i \otimes_{G_0} 
L_{2-i}(F^2)\half$ maps to $C_\pi(\tilde{z})=(0,0,C_\pi(z)\otimes {\bf e_1}F)\in T_2^F(G_1,\z\half)=
\bigoplus_{i=0}^2 P_i \otimes_{G_1} L_{2-i}(F^2)\half$. This image is a boundary map and therefore 
$d_3(C_\pi(x))=C_\pi(\tilde{z})$, where
\[
\begin{array}{c}
C_\pi(x)=(C_\pi(x_0),C_\pi(x_1),C_\pi(x_2),C_\pi(x_3))\in T_3^F(G_1,\z\half).
\end{array}
\]

By Corollary \ref{cor:delta}, $C_\pi$ acts trivially on $\im(\delta)$. Note that the following diagram is commutative
\[
\begin{tikzcd}[column sep=small]
& H_2(\Gamma,\z\half)^{\GG_F} \arrow[dl, "i_\ast"] \arrow[dr, "i_\ast' "] & \\
H_2(G_0,\z\half) \ar[rr, "{C_\pi}_\ast"] & &H_2(G_1,\z\half).
\end{tikzcd}
\]
Hence we may replace the element $C_\pi(\tilde{z})=(0,0,C_\pi(z)\otimes {\bf e_1}F)\in T_2^F(G_1,\z\half)$
with $\tilde{z}=(0,0,z\otimes {\bf e_1}F)\in T_2^F(G_1,\z\half)$.

Now under the map 
\[
\begin{array}{c}
T_2^F(G, \z\half[G/\Gamma]) \arr T_2^F(G, \z\half[G/G_0]) \oplus T_2^F(G, \z\half[G/G_1]),
\end{array}
\]
we have
\[
\tilde{z}\otimes \Gamma\mapsto
(\tilde{z}\otimes G_0, \tilde{z}\otimes G_1).
\]
By the calculation above, this is the boundary of
\[
\begin{array}{c}
(x\otimes G_0, C_\pi(x)\otimes G_1) \in T_3^F(G, \z\half[G/G_0]) \oplus T_3^F(G, \z\half[G/G_1]).
\end{array}
\]
Under the map 
\[
\begin{array}{c}
T_3^F(G, \z\half[G/G_0]) \oplus T_3^F(G, \z\half[G/G_1]) \arr T_3^F(G, \z\half),
\end{array}
\]
this elements maps to $\Omega:=C_\pi(x)-x$. By construction the cycle $\Omega$ represent an element of $H_3(G,\z\half)$,
which maps to $\omega \in H_2(\Gamma,\z\half)$ under the connecting map $\delta$. Now under the map
\[
\begin{array}{c}
H_3(\SL_2(F),\z\half) \arr  {\RP}_1(F)\half,
\end{array}
\]
the homology class of $\Omega$ maps to $\lan \pi \ran \alpha-\alpha\in \RB(F)\half$. By definition 
$\delta_\pi(\lan \pi \ran \alpha-\alpha)=\alpha$.  Therefore
\[
d_{3,0}^3\circ\Delta_\pi(\Omega)=d_{3,0}^3(\alpha)=\omega+K=p\circ\delta(\Omega).
\]
Now if $\Omega'$  is another element of $H_3(\SL_2(F),\z\half)$ thats map to $\omega$ by $\delta$, 
then $\Omega'=\Omega+\lan \pi \ran j_\ast (y)$ for some $y \in H_3(\SL_2(A),\z\half)$. Now using the 
commutativity of the left side square in our diagram, it is easy 
to see that
\[
d_{3,0}^3\circ\Delta_\pi(\Omega')=d_{3,0}^3(\alpha)=\omega+K=p\circ\delta(\Omega').
\]
This completes the proof of the proposition.
\end{proof}

\begin{cor}\label{exact4}
If $k$ is sufficiently large, then we have the exact sequence
\[
\begin{array}{c}
H_3(\SL_2(A),\z\half) \!\arr\!\! \RP_1(k)\half\! \overset{d}{\arr} 
\!H_2(\Gamma_0(\mmm_A),\z\half)\!\! \overset{i_\ast}{\arr}\! H_2(\SL_2(A),\z\half) \!\arr\! I^2(k)\half \!\arr \!0
\end{array}
\]
where $d$ is given by the formula
\[
d(g(a)) ={\bf c}\bigg({\mtxx {1-a} 0 0 {(1-a)^{-1}} }, {\mtxx a 0 0 {a^{-1}}} \bigg).
\]
\end{cor}
\begin{proof}
Set $\Gamma=\Gamma_0(\mmm_A)$. By Proposition \ref{main-diagram}  we have the commutative diagram
\[
\begin{tikzcd}
H_3(\SL_2(A),\z\half)\!\! \ar[r, "{\lan\pi\ran j_\ast}"] \ar[d, "="] &\!\!
\displaystyle\frac{H_3(\SL_2(F),\z\half)}{j_\ast H_3(\SL_2(A),\z\half)} 
\!\ar[r, "\delta "] \ar[d, "\overline{\Delta}_\pi "] &\!\!H_2(\Gamma,\z\half) 
\ar[d, two heads] \ar[r, "i_\ast"]\ar[d] &\!\!H_i(\SL_2(A),\z\half) \ar[d, "="]\\
H_3(\SL_2(A),\z\half)\!\! \ar[r] & \!\! \RP_1(k)\half \!\ar[r, "{d_{3,0}^3(A,k)} "] & 
\!\!E_{0,2}^3(A,k)\half\ar[r, "i_\ast"]&\!\!H_i(\SL_2(A),\z\half) ,
\end{tikzcd}
\]
which implies that $E_{0,2}^3(A,k)\half\simeq H_2(\Gamma,\z\half)$. Now
the claim follows from Proposition~\ref{exact2}.
\end{proof}

The following is Theorem B in the introduction.

\begin{thm}\label{barP}
Let $A$ be a discrete valuation ring with sufficiently large residue field $k$. Then the inclusion $\Gamma_0(\mmm_A)\arr \SL_2(A)$
induces the exact sequence od $\RR_A$-modules
\[
\begin{CD}
0\arr \overline{\PP}(k)\half \arr H_2(\Gamma_0(\mmm_A),\z\half) \arr H_2(\SL_2(A),\z\half)\arr I^2(k)\half \arr 0,
\end{CD}
\]
where the left homomorphism is given by 
$\displaystyle [a] \mt \frac{1}{2}{\bf c}\bigg({\mtxx {1-a} 0 0 {(1-a)^{-1}} }, {\mtxx a 0 0 {a^{-1}}} \bigg)$.
\end{thm}

\begin{proof}
By the results above, the diagram
\[
\begin{tikzcd}
&\RP_1(k)\half\ar[d, two heads]\ar[rd, "d"]&\\
0\ar[r]&\overline{\PP}(k)\half\ar[r]&H_2(\Gamma_0(\mmm_A),\z\half)\\
\end{tikzcd}
\]
commutes, where the vertical arrow sends $g(a)$ to $2[a]$. 
\end{proof}

If $k$ is finite, then $I^{2}(k)=0$ and we deduce:

\begin{cor}
If $k$ is a sufficiently large finite field, then we have the exact sequence
\[
\begin{CD}
0\arr \overline{\PP}(k)\half \arr H_2(\Gamma_0(\mmm_A),\z\half) \arr H_2(\SL_2(A),\z\half)\arr  0.
\end{CD}
\]
\end{cor}

\section{The case of global fields}\label{global}

In this final section we will study the group $\overline{\PP}(k)$ when $F$ is a global field.

\subsection{The {\it e}-invariant  of a field}

Let $E$ be a separably closed field. 
The group $\Aut(E)$ acts naturally on $\mu_E$, the group of roots of unity in $E$, where $\sigma\in \Aut(E)$
sends $\zeta$ to $\sigma(\zeta)$. This action gives a surjective map $\Aut(E)\arr \Aut(\mu_E)$. Note that
$\mu_E\simeq \q/\z$ if $\char(E)=0$ and $\mu_E\simeq (\q/\z)[\frac{1}{p}]$ if $\char(E)=p>0$. 

For $i\in\z$, we define $\mu_E(i)$ as $\mu_E$ turned into $\Aut(E)$-module by letting $\sigma\in \Aut(E)$ acts as 
\[
\zeta\mapsto \sigma^i(\zeta).
\]
On the other hand, $\Aut(E)$ acts on $K_n(F)$ for any $n\geq 0$: $\sigma:E \arr E$ induces the isomorphism 
$\sigma_\ast:K_n(E)\arr K_n(E)$ and for any $x \in K_n(E)$ the action is given by $\sigma.x=\sigma_\ast(x)$.

If $n$ is odd, then $K_n(E)_\tor\simeq \mu_E$ (if $n$ is even, $K_n(E)$ is uniquely divisible). 
Clearly the action of $\Aut(E)$ on $K_n(E)$
induces an action of $\Aut(E)$ on $K_n(E)_\tor$. It is known that for $i>0$, $K_{2i-1}(E)_\tor$ is isomorphic 
to $\mu_E(i)$ as $\Aut(E)$-module \cite[Proposition 1.7.1, Chap. VI]{weibel2013}.

Let $F$ be a field and $F^\sep$ be its separable closure. Let $G_F:=\Gal(F^\sep/F)$. Since
the natural map $K_n(F)\arr K_n(F^\sep)$ is a homomorphism of $G_F$-modules with $G_F$ acting trivially on $K_n(F)$
it follows that there is a natural map $K_n(F)_\tor \arr (K_n(F^\sep)_\tor)^{G_F}$. The $e$-invariant of 
($K_{2i-1}$-group of ) $F$ is the composition
\[
K_{2i-1}(F)_\tor \larr (K_{2i-1}(F^\sep)_\tor)^{G_F} \overset{\simeq}{\larr} \mu_{F^\sep}(i)^{G_F}.
\]
If $\mu_{F^\sep}(i)^{G_F}$ is  finite, it is cyclic and we write $w_i(F)$ for its order. Thus
\[
\mu_{F^\sep}(i)^{G_F}\simeq \z/w_i(F).
\]
We will need the $e$-invariant of fields for the third $K$-group. Thus from now on we will discuss only this special case.

\begin{exa} (Finite Fields) For a finite field $\F_q$, $w_2(\F_q)=q^2-1$. One the other hand $K_3(\F_q)\simeq \z/(q^2-1)$.
On can show that the $e$-invariant 
\[
e:K_3(\F_q)_\tor \larr \mu_{\overline{\F}_q}(2)^{G_{\F_q}}
\]
is an isomorphism \cite[Example 2.1.1, Chap. VI]{weibel2013}. 
\end{exa}

From now on we assume that $A$ is a discrete valuation ring, with field of fractions $F$ and residue field $k$.
Let $v=v_A:\ff \arr \z$ be the valuation associated to $A$.

\subsection{The {\it e}-invariant  of local and global fields and the indecomposable \texorpdfstring{$K_3$}{Lg}}

Let $p$ be a prime number. For any torsion abelian group $G$, we let $G[p']$ denote the subgroup 
of elements of order prime to $p$. In fact $G[p']\simeq G\otimes \z[\frac{1}{p}]\simeq G[\frac{1}{p}]$. Thus
we have a canonical surjective map 
\[
G \arr G[p'].
\]

A (non-Archimedean ) local field $E$ is a field which is complete in respect with a discrete valuation 
$v:E^\times \arr \z$ and has finite residue field. In this case the discrete valuation ring is a complete ring. 
It is a classical result that a local field is either a finite extension of the rational p-adic field $\q_p$ or 
is isomorphic to $\F_q((x))$ for some finite field $\F_q$ (see \cite{serre1979}).

\begin{lem}[Local Fields]\label{localfields}
If $E$ is a local field, then we have the isomorphism
\[
e: K_3^\ind(E)[p']_\tor \overset{\simeq}{\larr}\mu_{E^\sep}(2)[p']^{G_E},
\]
where $p$ is the characteristic of the residue field.
\end{lem}
\begin{proof}
If $E$ is finite over $\q_p$  with residue field $\F_q$, then 
\[
w_2(E)=p^nw_2(\F_q)=p^n(q^2-1)
\]
for some $n\geq 0$ (by Hensel's lemma). Moreover the map
\[
e:K_3(E)_\tor \larr \mu_{E^\sep}(2)^{G_E}\simeq \z/w_2(E)
\]
is surjective and induces an isomorphism on $\ell$-primary torsion subgroups
$K_3(E)(l)\simeq \z/w_2^{(\ell)}(E)$ for any prime $\ell\neq p$ 
\cite[Example 2.3.2, Chap. VI]{weibel2013}.

If $E=\F_q((t))$, then $K_3(E)_\tor \simeq K_3(\F_q)$ \cite[Theorem 7.2, Chap. VI]{weibel2013}. This 
shows that $w_2(E)=w_2(\F_q((t)))=q^2-1$ and the $e$-invariant
\[
e:K_3(E)_\tor \larr \mu_{E^\sep}(2)^{G_E}\simeq \z/w_2(E)
\]
is an isomorphism.

Since for any local filed $E$, $K_3^M(E)$ is uniquely divisible 
\cite[Proposition 7.1, Chap. VI]{weibel2013}, the $e$-invariant 
factors through $K_3^\ind(E)_\tor$. Therefore we get the desired isomorphism.
\end{proof}

\begin{cor}\label{cor-localfields}
Let $F$ be local. If $\char(k)=p$, then we have the commutative diagram
\[
\begin{tikzcd}
K_3^\ind(A)[p']_\tor \ar[r, "\simeq"]\ar[d] & \mu_{F^\sep}(2)[p']^{G_F}\ar[d,"\simeq"]\\
K_3^\ind(k) \ar[r, "\simeq"]                        & \mu_{\bar{k}}(2)^{G_k},
\end{tikzcd}
\]
where the vertical maps are induced by the quotient map $A\arr k$. Moreover all the maps involved are isomorphism.
\end{cor}

\begin{lem}[Global Fields]\label{globalfields}
For any global field $E$, the $e$-invariant
factors through the group $K_3^\ind (E)$ and gives  the isomorphism
\[
e: K_3^\ind(E)_\tor \overset{\simeq}{\larr} \mu_{E^\sep}(2)^{G_E}\simeq \z/w_2(E).
\]
\end{lem}
\begin{proof}
If $\char(E)=0$, then $E$ is an algebraic number field. For this the claim follows from 
\cite[Corollary 5.3, Chap. VI]{weibel2013}.

If $\char(E)=p>0$, then $E$ is finite over a field of the
form $\F_q(t)$, where $q$ is $p$-power.	In this case $K_3^M(E)=0$ \cite{bass-tate1973} and
the natural map $K_3(\F_q)\arr K_3(E)$ is an isomorphism \cite[Theorem 6.8, Chap. VI]{weibel2013}.
Therefore $K_3(E)=K_3^\ind(E)$ and clearly the e-invariant is an isomorphism. 
\end{proof}


\subsection{The image of the map \texorpdfstring{$K_3^\ind(F)\arr \PP(k)$}{Lg}}

\begin{thm}\label{split}
Let $F$ be a global field such that $\char(k)\nmid w_2(F)$. Then
\par {\rm (i)} There is a natural splitting of the inclusions $K_3^\ind(F)_\tor \arr K_3^\ind(F)$, call it
\[
p_F: K_3^\ind(F)\arr K_3^\ind(F)_\tor.
\]
\par {\rm (ii)} The map $K_3^\ind(F) \simeq K_3^\ind(A)\arr K_3^\ind(k)$ factors through $p_F$.
\par {\rm (iii)} The image of $K_3^\ind(F)\simeq K_3^\ind(A)\arr K_3^\ind(k)$ is cyclic of order $w_2(F)$.
\end{thm}
\begin{proof}
Consider the commutative diagram
\[
\begin{tikzcd}
K_3^\ind(F)_\tor \ar[rrr, "e_F"]\ar[dddd]\ar[dr, hook]            &  &               & \mu_{F^\sep}(2)^{G_F}\ar[dddd, "i", hook]\\
                                                           & K_3^\ind(F) \ar[dd] \ar[dr, "\pi"]  &                                             &    \\
                                                           &                                                & K_3^\ind(k) \ar[ddr, "\alpha"]                & \\
                                                           & K_3^\ind(F_v) \ar[ru, "\pi_v"]          &                                                  & \\
K_3^\ind(F_v)_\tor[p'] \ar[rrr, "e_{F_v}'"]\ar[ru, hook]  &      &    & \mu_{F_v^\sep}(2)[p']^{G_{F_v}}            
\end{tikzcd}
\]
where $F_v$ is the completion of $F$ with respect to the valuation and $p=\char(k)$.
Note that $\alpha$ is coming from Corollary \ref{cor-localfields} and $\pi$ is the composite
$K_3^\ind(F)\simeq K_3^\ind(A)\arr K_3^\ind(k)$.
Moreover the maps $e_F$, $e_{F_v}$ and $\alpha$ are isomorphism by Lemmas \ref{globalfields}, \ref{localfields}
and Corollary \ref{cor-localfields}, respectively. 

(i) From the commutativity of the above diagram one can take $p_F=e_F^{-1}\circ i^{-1}\circ \alpha\circ \pi$.
\par (ii) By condtruction $\pi=\alpha^{-1}\circ  i\circ e_F\circ p_F$. Hence it factors through $p_F$.
\par (iii) Since $\pi$ factors throught $p_F$, we have 
\begin{align*}
\im(K_3^\ind(A)\arr K_3^\ind(k))& =\im(K_3^\ind(A)_\tor\arr K_3^\ind(k))\\
&\simeq  \im ( \mu_{F^\sep}(2)^{G_F}\overset{i}{\harr} \mu_{F_v^\sep}(2)[p']^{G_{F_v}} ).
\end{align*}
This completes the proof.
\end{proof}

\begin{prp}\label{A-to-k}
Let $F$ be a global field such that $\char(k)\nmid w_2(F)$.
Then the image of the natural map  $K_3^\ind(A) \arr \PP(k)$ is cyclic of order $\gcd\Big(w_2(F),(|k|+1)/2\Big)$ if 
$\char(k)\neq 2$ and is $\gcd\Big(w_2(F),|k|+1\Big)$ if $\char(k)=2$.
\end{prp}
\begin{proof}
The map $K_3^\ind(A) \arr \PP(k)$ factors through $K_3^\ind(k)$. Let $\eta$ be the map $K_3^\ind(A) \arr K_3^\ind(k)$.
By Theorem \ref{split}, the image of $\eta$ is cyclic of order $w_2(F)$. By Theorem \ref{bloch-wigner}, we have the 
Bloch-Wigner exact sequence 
\[
0 \arr \tors(k^\times, k^\times)^\sim \overset{\beta}{\arr} K_3^\ind(k) \arr \PP(k) \arr S^2_\z(k^\times) \arr K_2(k)\arr 0,
\]
where $\tors(k^\times, k^\times)^\sim$ is the extension of $\tors(k^\times, k^\times)$ by $\z/2$ if $\char(k)\neq 2$
and is isomorphic to $\tors(k^\times, k^\times)$ if $\char(k)=2$.

Let $k=\F_q$. If $\char(k)\neq 2 $, then $\tors(k^\times, k^\times)^\sim$ is a cyclic group of order $2(q-1)$. 
On the other hand $K_3^\ind(k)=K_3(k)$ is cyclic of order $q^2-1$. Thus $K_3^\ind(k)/\im(\beta)$ is cyclic of order
$(q+1)/2$.

We need to calculate the composite
\[
\im(\eta) \arr K_3^\ind(k)/\im(\beta) \harr \PP(k).
\]
Since the image of $\im(\eta) \arr K_3^\ind(k)/\im(\beta)$ is cyclic of order $\gcd\Big(w_2(F),(|k|+1)/2\Big)$
we are done. If $\char(k)\neq 2 $, then $\tors(k^\times, k^\times)$ is a cyclic group of order $q-1$. The rest of proof 
is similar to the above argument.
\end{proof}

\subsection{The second homology of \texorpdfstring{$\Gamma_0(\mmm_A)$}{Lg} of a DVR in a global field}

For a natural number $n$, let $n'$ be the odd part of $n$, i.e. $n'=n/2^r$, where $2^r$ is highest power of $2$ 
thats divide $n$. By combining Lemma \ref{barP}, Theorem \ref{mainthm} and Proposition \ref{A-to-k} we obtain 
the following result.

\begin{thm}\label{cong-global}
Let $F$ be a global field such that $\char(k)\nmid w_2(F)$. If  $k$ is sufficiently large then we have the exact sequence
\[
\begin{CD}
0\arr \overline{\PP}(k)\half \arr H_2(\Gamma_0(\mmm_A),\z\half) \arr H_2(\SL_2(A),\z\half)\arr 0,
\end{CD}
\]
where $\PP(k)\half$ is cyclic of order $(|k|+1)'$ and $\overline{\PP}(k)\half$ is the quotient of $\PP(k)\half$ by the unique
cyclic subgroup of order $\gcd\Big(w_2(F),|k|+1\Big)'$. Moreover the left homomorphism is given by 
$\displaystyle [a] \mt \frac{1}{2}{\bf c}\bigg({\mtxx {1-a} 0 0 {(1-a)^{-1}} }, {\mtxx a 0 0 {a^{-1}}} \bigg)$.
\end{thm}

Let $F=\q$ equipped with the $p$-adic valuation $v_p$. The residue field is $\F_p$ and $w_2(\q)=24$
\cite[Example 2.1.2, Chap. VI] {weibel2013}. If  $p\neq 2, 3$, then $24|p^2-1=|K_3^\ind(\F_p)|$.  We have
\[
\gcd(w_2(\q), p+1)'=\gcd(3, p+1)=
\begin{cases}
3 & \text{if $3\mid p+1$}\\
0 & \text{if $3\nmid p+1$}.
\end{cases}
\]
From the previous theorem and  the fact that $c_{\F_p}$ is of order $\gcd(6, (p+1)/2)$ in $\PP(\F_p)$ 
\cite[Lemma 7.11]{hutchinson-2013}, we obtain the following corollary.
\begin{cor}
For any prime number $p\geq 11$, we have the exact sequence
\[
\begin{array}{c}
0\arr \overline{\PP}(\F_p)\half \arr H_2(\Gamma_0(\z_{(p)}),\z\half) \arr H_2(\SL_2(\z_{(p)}),\z\half)\arr 0,
\end{array}
\]
where $\overline{\PP}(\F_p)={\PP}(\F_p)/\lan c_{\F_p} \ran$ if $3\mid p+1$ and $\overline{\PP}(\F_p)={\PP}(\F_p)$
if $3\nmid p+1$.
\end{cor}

As a final  application of the results above, we show:

\begin{prp}\label{prop:zp} 
Let $p\geq 11$ be a prime. Let $\bar{H}_3(\SL_2(\z_{(p)}),\z) $ denote the image of the natural homomorphism 
$H_3(\SL_2(\z_{(p)}),\z)\to H_3(\SL_2(\q),\z)$ 
and let $\bar{H}_3(\SL_2(\z_{(p)}),\z)_0$ denote
 $\ker(\bar{H}_3(\SL_2(\z_{(p)}),\z)\to K_3^\ind(\q))$. Then 
\[
\begin{array}{c}
\bar{H}_3(\SL_2(\z_{(p)}),\z\half)_0 \simeq \bigoplus_{q\not=p}\PP(\F_p)\half,
\end{array}
\]
where the sum runs over all primes $q$ not equal to $p$. 
\end{prp}

\begin{proof} We have a commutative diagram with exact rows, whose middle column is exact by Theorem \ref{mainthm}:
\[
\begin{tikzcd}
0\ar[r]&\bar{H}_3(\SL_2(\z_{(p)}),\z\half)_0\ar[d]\ar[r]&\bar{H}_3(\SL_2(\z_{(p)}),\z\half)\ar[d]\ar[r]&K_3^\ind(\q)\half\ar[r]\ar[d, "="]&0\\
0\ar[r]&H_3(\SL_2(\q),\z\half)_0\ar[r]\ar[d, "{\delta_p}"]&H_3(\SL_2(\q),\z\half)\ar[r]\ar[d, "{\Delta_p}"]&K_3^\ind(\q)\half\ar[r]&0\\
&\PP(\F_p)\half\ar[r, "="]\ar[d]&\PP(\F_p)\ar[d]\half&&\\
&0&0&&\\
\end{tikzcd}
\]
Thus $\bar{H}_3(\SL_2(\z_{(p)}),\z\half)_0\cong \ker(\Delta_p:\bar{H}_3(\SL_2(\q),\z\half)_0\to \PP(\F_p)\half)$. But this is equal to 
$\ker(\Delta'_p:\bar{H}_3(\SL_2(\q),\z\half)_0\to \PP(\F_p)\half)$ by Proposition \ref{2times} above.  However 
$\ker{\Delta'_p}\simeq \bigoplus_{q\not=p}\PP(\F_p)\half$ by the main theorem of \cite{hutchinson2019} (where $\Delta'_p$ is denoted $S_p$). 
\end{proof}


\end{document}